\documentclass[11pt,leqno]{article}

\usepackage[T1]{fontenc}
\usepackage{amssymb}
\usepackage{amsmath}
\usepackage[title, toc]{appendix}
\usepackage{amsthm}  
\usepackage{tikz-cd}  
\usepackage{tikz,stackengine}  
\usepackage{enumerate} 
\usepackage{enumitem}  
\usepackage{pdfpages}  
\usepackage{hyperref}
\usepackage{wasysym}  
\usepackage{mathtools} 
\usepackage{stmaryrd} 
\usepackage{graphicx}
\usepackage{xcolor}
\usepackage{mathrsfs}
\usepackage{titlesec}
\usepackage{ytableau}
\usepackage{ bbold }
\usepackage[matrix,arrow,curve]{xy}
\usepackage{xcolor}
\hypersetup{
    colorlinks,
    linkcolor={blue!50!black},
    citecolor={blue!50!black},
    urlcolor={blue!80!black}
}
\DeclareSymbolFont{extraitalic} {U}{zavm}{m}{it}
\DeclareMathSymbol{\stigma}{\mathord}{extraitalic}{168}

\newtheorem*{theorem*}{Theorem}
\newtheorem{theorem}{Theorem}[subsection]
\newtheorem{proposition}[theorem]{Proposition}
\newtheorem{definition}[theorem]{Definition}
\newtheorem*{proposition*}{Proposition}

\newtheorem{lemma}[theorem]{Lemma}
\newtheorem{corollary}[theorem]{Corollary}

\theoremstyle{remark}
\newtheoremstyle{boldremark}
{\dimexpr\topsep/2\relax} 
{\dimexpr\topsep/2\relax} 
{}          
{}          
{\bfseries} 
{.}         
{.5em}      
{}          
\theoremstyle{boldremark}\newtheorem{bremark}[theorem]{Remark}
\newtheoremstyle{bolddefinition}
{\dimexpr\topsep/2\relax} 
{\dimexpr\topsep/2\relax} 
{}          
{}          
{\bfseries} 
{.}         
{.5em}      
{}          
\theoremstyle{bolddefinition}
\newtheorem*{conjecture*}{Conjecture}
\makeatletter
\newcommand*\leftdash{\rotatebox[origin=c]{-45}{$\dabar@\dabar@\dabar@$}}
\newcommand*\rightdash{\rotatebox[origin=c]{45}{$\dabar@\dabar@\dabar@$}}
\newcommand*\mondash[1]{\rotatebox[origin=c]{45}{$\dabar@\dabar@\dabar@$}{\kern-1.5ex\raisebox{-.5ex}{}_{#1}}\kern.5ex}
\makeatother
\newcommand{\Fr}{\operatorname{Fr}}

\newcommand{\Qlbar}{\overline{\mathbb{Q}}_{\ell}}

\newcommand{\Fqbar}{\overline{\mathbb{F}}_q}
\newcommand{\Fq}{\mathbb{F}_q}

\newcommand{\Ge}{G_{\operatorname{enh}}}
\newcommand{\Ta}{T_{\operatorname{adj}}}

\newcommand{\yy}{\mathcal{Y}}

\newcommand{\barGe}{\overline{G_{\operatorname{enh}}}}
\newcommand{\obarGe}{\overset{\circ}{\overline{G_{\operatorname{enh}}}}}
\newcommand{\barTa}{\overline{T_{\operatorname{adj}}}}
\newcommand{\F}{\mathcal{F}}
\newcommand{\G}{\mathcal{G}}
\newcommand{\cO}{\mathcal{O}}
\newcommand{\hc}{\mathfrak{hc}}
\newcommand{\cc}{\underline{\mathbb{k}}}

\newcommand{\DD}{\hat{\Delta}}
\newcommand{\dd}{\hat{\delta}}
\newcommand{\NN}{\hat{\nabla}}

\newcommand{\T}{\hat{\mathcal{T}}}

\newcommand{\Spgr}{\Sigma}
\newcommand{\Hu}{\mathbf{1}}

\newcommand{\rsim}{\xrightarrow{
   \,\smash{\raisebox{-0.65ex}{\ensuremath{\scriptstyle\sim}}}\,}}
\newcommand{\cC}{\mathcal{C}}
\newcommand{\cD}{\mathcal{D}}
\newcommand{\cA}{\mathcal{A}}
\newcommand{\cB}{\mathcal{B}}
\newcommand{\cM}{\mathcal{M}}
\newcommand{\cZ}{\mathcal{Z}}

\newcommand{\pro}{\operatorname{pro}}
\newcommand{\M}{\mathscr{M}}
\newcommand{\hM}[2]{{}_{#1}\hat{\mathscr{M}}_{#2}}

\newcommand{\Perv}{{\mathcal{P}}}
\newcommand{\cPerv}{\hat{\mathcal{P}}}
\newcommand{\Hecke}[1]{\mathcal{H}^{(#1)}}
\newcommand{\cofree}[2]{#1-\operatorname{cofree}_{#2}}
\newcommand{\comod}[2]{#1-\operatorname{comod}_{#2}}

\newcommand{\cS}{\mathcal{S}}

\newcommand{\Hom}{\operatorname{Hom}}
\newcommand{\Spec}{\operatorname{Spec}}

\newcommand{\D}{\mathbb{D}}

\newcommand{\ustar}{\,\underline{\star}\,}

\newcommand{\cf}{\operatorname{cofree}}
\newcommand{\Id}{\operatorname{Id}}
\newcommand{\bistar}[1]{{\star}_{#1}}

\newcommand{\Ad}{\operatorname{Ad}}

\titleformat{\subsection}[runin]
{\normalfont\bfseries}{\thesubsection}{1em}{}

\titleformat{\subsubsection}[runin]
{\normalfont\bfseries}{\thesubsection}{1em}{}
\newcommand{\Addresses}{{
  \bigskip
  \footnotesize

  R.~Bezrukavnikov, \textsc{Department of Mathematics, Massachusetts Institute of Technology,
    Cambridge, Massachusetts, USA}\par\nopagebreak
  \textit{E-mail address}: \texttt{bezrukav@math.mit.edu}
  
    \medskip

  A.~Ionov, \textsc{Department of Mathematics, Boston College,
    Chestnut Hill, Massachusetts, USA}\par\nopagebreak
  \textit{E-mail address}: \texttt{ionov@bc.edu}

  \medskip

  K.~Tolmachov, \textsc{University of Edinburgh,
    Edinburgh, United Kingdom}\par\nopagebreak
  \textit{E-mail address}: \texttt{tolmak@khtos.com}
  
    \medskip

  Y.~Varshavsky, \textsc{Institute of Mathematics, Hebrew University, Givat-Ram, Jerusalem, Israel}\par\nopagebreak
  \textit{E-mail address}: \texttt{vyakov@math.huji.ac.il}

}}

\begin{document} \title{Equivariant derived category of a reductive group as a categorical center}

\author{Roman Bezrukavnikov, Andrei Ionov,\\ Kostiantyn Tolmachov, Yakov Varshavsky} \date{}
\maketitle \abstract{We prove that the adjoint equivariant derived category of a reductive group $G$ is equivalent to the appropriately defined monoidal center of the torus-equivariant version of the Hecke category. We use this to give new proofs, independent of sheaf-theoretic set up, of the fact that the Drinfeld center of the abelian Hecke category is equivalent to the abelian category of unipotent character sheaves; and of a characterization of strongly-central sheaves on the torus.}
\newpage
\tableofcontents
\section{Introduction}
Let $G$ be a connected reductive group, either over $\mathbb{C}$ or split over a finite field. In this paper, we study the equivariant, with respect to the adjoint aciton of $G$ on itself, constructible derived category of sheaves on $G$, which we denote $D^b(G/_{\Ad}G)$, not specifying the precise sheaf-theoretic setting for now. The category $D^b(G/_{\Ad}G)$ is well-known to be a braided monoidal category. We relate this category to the appropriately defined center of another triangulated monoidal category $\Hecke{1}$, which we call the $T$-equivariant Hecke category (for the adjoint action of a maximal torus $T \subset G$).

Similarly, we relate the derived category $D^b_{\mathfrak{C}}(G) \subset D^b(G/_{\Ad}G)$ of unipotent character sheaves on $G$ to the appropriately defined center of the monodromic $T$-equivariant Hecke category with monodromic parameter $\lambda$, $\Hecke{1}_{\lambda}$. 

The idea to relate the category of character sheaves to the categorical center of the monoidal Hecke category is not new. In \cite{bezrukavnikovCharacterDmodulesDrinfeld2012}, first named author, Finkelberg and Ostrik proved that the abelian category $\mathfrak{C}$ of unipotent character sheaves is equivalent to the Drinfeld center of the abelian category of Harish-Chandra D-modules. We reprove their result and extend it to other sheaf-theoretic settings.

In \cite{ben-zviCharacterTheoryComplex2009}, authors prove that the appropriately defined $\infty$-category of unipotent character sheaves is the categorical center of both equivariant and monodromic versions of the $\infty$-Hecke category in the setting of $D$-modules. While this work was in final stages of preparation, preprint \cite{holicharactersheveshomfly2023} appeared, where the Drinfeld center of the Hecke category is computed in a general $\infty$-categorical setting. Our approach is different in that we don't leave the world of triangulated categories (except when considering the statement about abelian Drinfeld center). We deduce our results from a statement valid for all equivariant sheaves on the group (Theorem \ref{sec:hecke-categ-char}). We were informed by David Ben-Zvi that a close statement in the $\infty$-categorical D-module setting follows from the results of \cite{ben-zvigunninghamorem2018}. 

In the upcoming work, second and third named authors, together with the authors of \cite{holicharactersheveshomfly2023}, are planning to use this fact to define a mixed Koszul self-duality for the derived category of character sheaves. This has applications, in particular, to symmetries in Khovanov-Rozansky homology theories, and the generalizations of those theories, as in \cite{gorskyHilbertSchemesIfication2022}. Two-periodic version of such a Koszul duality equivalence in characteristic 0 appears in \cite{ben-zviCharacterTheoryComplex2009}.

Finally, we use our results to give a new proof of a variant of the vanishing result proved in  \cite{chenConjectureBravermanKazhdan2022}, eliminating a restriction on the characteristic.  

We also mention that the description of the asymptotic abelain category of character sheaves as a categorical center was obtained in \cite{lusztigTruncatedConvolutionCharacter2014}, \cite{lusztigNonunipotentCharacterSheaves2015}. 

\subsection{Organization of the paper.} In Section \ref{sec:equiv-deriv-categ} we define the Hecke categories we are working with, the variant of the categorical center we are using, and state our main result for triangulated categories.

In Section \ref{sec:sheaves-group-as} we recall the notion of a separable comonad and relevant results regarding comodules over comonads in triangulated categories, following \cite{balmerSeparabilityTriangulatedCategories2011}.

In Section \ref{sec:dbgg-as-categorical} we prove our main result for triangulated categories.

In Section \ref{sec:character-sheaves-as} we recall some results about monodromic categories and character sheaves, state the variant of the main result for the derived category of character sheaves and prove the result about the Drinfeld center of the abelian Hecke category.
We also study vanishing properties for central sheaves on a torus.

In the Appendix~\ref{sec:sheav-vinb-semigr}, following an idea of V. Drinfeld, we give a proof of the t-exactness of convolution on the horocycle space using the properties of convolution on the Vinberg semigroup attached to $G$.   

\subsection{Acknowledgments.} We would like to thank Xin Jin, Anton Mellit and Grigory Papayanov for helpful discussions. We would like to thank Colton Sandvik for pointing out a missing assumption on characteristic in the earlier version. R.B and Y.V were partially supported by the BSF grant 2020189. A.I. was funded by RFBR, project number 20-01-00579. K.T. is supported by the EPSRC programme grant EP/R034826/1. Y.V was partially supported by the ISF grant 2019/21. 

\section{Equivariant derived category of $G$ as a categorical center.}
\label{sec:equiv-deriv-categ}
\subsection{Notations.} Fix primes $\ell \neq p$. Let $q = p^n$ for $n \in \mathbb{Z}_{>0},$ and write $\Fq$ for the finite field with $q$ elements.

We will work with the following sheaf-theoretic settings.

\begin{enumerate}[label=(\Roman*),ref=(\Roman*)]
\item\label{set:c} For a stack $X$ defined over $\mathbb{C}$, let $D^b(X)$ be the bounded algebraically constructible derived category of sheaves on $X$ with analytic topology, with coefficients in $\mathbb{k} = \mathbb{C}$ (more generally, any field of characteristic 0).
\item\label{set:dmod} For a stack $X$ defined over $\mathbb{C}$, let $D^b(X)$ be the bounded derived category of holonomic $D$-modules on $X$.
\item\label{set:p} For a stack $X$ defined over $\mathbb{C}$, let $D^b(X)$ be the bounded constructible derived category of sheaves on $X$ with coefficients in a field $\mathbb{k}$ of characteristic $\ell$.

  In settings \ref{set:c},\ref{set:dmod}, \ref{set:p} we use the definition of sheaves on stacks as in \cite{bernsteinEquivariantSheavesFunctors2006}. 

\item\label{set:adic} For a stack $X$ defined over an algebraically closed field of characteristic $p > 0$, let $D^b(X)$ be the bounded
constructible derived category of $\Qlbar$-sheaves on $X$.
\item\label{set:mixed} For a stack $X$ defined over $\Fq$, let $D^b(X)$ be the bounded constructible
mixed derived category of $\Qlbar$-sheaves on $X$.  

  In settings \ref{set:adic},\ref{set:mixed}, we use the definition of sheaves on stacks as in \cite{laszloSixOperationsSheaves2008a}.

\end{enumerate}

 We define $\mathbb{k}$ to be the field $\mathbb{C}$ in the settings \ref{set:c} and \ref{set:dmod}, any field of characteristic $\ell$ in the setting \ref{set:p} and $\Qlbar$ in the settings \ref{set:adic}, \ref{set:mixed}. We write $\operatorname{pt}$ for the spectrum of the field of definition of our stacks in each of the cases. We write $\cc_X$ for the constant sheaf on $X$ (D-module $\cO_X$ in \ref{set:dmod}).   

\subsection{Hecke categories.} 
\label{sec:hecke} Let $G$ be a split connected reductive group defined over $\Fq$ (in cases \ref{set:adic}, \ref{set:mixed}) or a connected reductive group over $\mathbb{C}$ (in all other cases). Fix a Borel subgroup $B \subset G$, a split maximal torus $T \subset B$. Let $U$ be the unipotent radical of $B$. Let $X = G/B$ be the flag variety of $G$, $Y = G/U$ be the basic affine space of $G$. The natural projection $Y \to X$ is a $T$-torsor with respect to the right multiplication action of $T$ on $G/U$.

Let $W$ be the Weyl group of $G$. In the case \ref{set:p} we always assume that the characteristic of the field of coefficients does not divide the order $|W|$ of $W$.

Consider the right diagonal action of $T$ on $Y^2$. Let
\[
  \yy = (Y \times Y)/T
\]
and let $G$ act on $\yy$ by the diagonal left multiplication. The categories $D^b(\yy),$ $D^b(G\backslash \yy)$ are equipped with a monoidal structure via the following formula. Consider the diagram
\begin{equation}
  \label{eq:7}
  \begin{tikzcd}
    & Y^3/T \arrow[ld,"p_{12}"']\arrow[rd,"p_{23}"]\arrow[r,"p_{13}"]  & \yy \\
\yy &                                      & \yy 
  \end{tikzcd}
\end{equation}
Here $p_{ij}$ stands for the projection $Y^3 \to Y^2$ along factors $i,j$ and $T$ acts on $Y^3$ diagonally on the right, while $G$ diagonally on the left. We define the convolution product $-\star_1-$ on $D^b(\yy)$ or $D^b(G\backslash \yy)$ as
\[
  \mathcal{A}\star_1\mathcal{B} := p_{13!}(p_{12}^*\mathcal{A}\otimes p_{23}^*\mathcal{B}).
\]
Denote $\Hecke{1}:=D^b(G\backslash \yy)$. The category $\Hecke{1}$ has an alternative description as a (mixed in the relevant settings)  derived category of $\Ad T$-equivariant sheaves on $U\backslash G/U$. Let $\Hu$ stand for the monoidal unit of this category.

Consider the space $\yy^{(2)} = Y^4/T^2$ where the right $T^2$-action on $Y^4$ is defined by
\begin{equation}
  \label{eq:6}
   (x_1U,x_2U,x_3U,x_4U)\cdot (t,z) = (x_1tU,x_2zU,x_3zU,x_4tU).
\end{equation}
Let $G^2$ act on $\yy^{(2)}$ via the formula
\[
  (g,h)\cdot (x_1U,x_2U,x_3U,x_4U) = (gx_1U,gx_2U,hx_3U,hx_4U).
\]
The categories $D^b(\yy^{(2)}), D^b(G^2\backslash\yy^{(2)})$ are equipped with monoidal structures via the following formula. Consider the diagram
\[
  \begin{tikzcd}
    & Y^6/T^3 \arrow[ld,"p_{1256}"']\arrow[rd,"p_{2345}"]\arrow[r,"p_{1346}"]  & \yy^{(2)} \\
\yy^{(2)} &                                      & \yy^{(2)} 
  \end{tikzcd}
\]
Here $p_{ijkl}$ stands for the projection $Y^6 \to Y^4$ along factors $i,j,k,l$ and $T^3$ acts on $Y^6$ on the right according to the formula
\begin{multline*}
  (x_1U,x_2U,x_3U,x_4U,x_5U,x_6U)\cdot  (u,v,w)=\\
 =(x_1uU,x_2vU,x_3wU,x_4wU,x_5vU,x_6uU). 
\end{multline*} 
while $G^2$ acts on $Y^6$ on the left according to the formula
\begin{multline*}
  (g,h)\cdot (x_1U,x_2U,x_3U,x_4U,x_5U,x_6U) =\\
 =(gx_1U,gx_2U,gx_3U,hx_4U,hx_5U,hx_6U). 
\end{multline*} 
Define the convolution product $-\star_2-$ on $D^b(\yy^{(2)})$ or $D^b(G^2\backslash\yy^{(2)})$ as
\begin{equation}
  \label{eq:8}
  \mathcal{A}\star_2\mathcal{B} := p_{1346!}(p_{1256}^*\mathcal{A}\otimes p_{2345}^*\mathcal{B}).
\end{equation}
Denote $\Hecke{2}:=D^b(G^2\backslash\yy^{(2)})$. 

The category $\Hecke{1}$ is a module category for the monoidal category $\Hecke{2}$. The action is given by the ``two-sided convolution'' defined via the following formula. Consider the diagram 
\[
  \begin{tikzcd}
    & \yy^{(2)} \arrow[ld,"p_{23}"']\arrow[rd,"p_{14}"] & \\
\yy &                                      & \yy 
  \end{tikzcd}
\]
For $\mathcal{B} \in \Hecke{2},\mathcal{A}\in\Hecke{1},$ we define
\[
  \mathcal{B}\bowtie\mathcal{A} = p_{14!}(\operatorname{For}^{G^2}_G(\mathcal{B}) \otimes p_{23}^*(\mathcal{A})).
\]
Here $\operatorname{For}^{G^2}_G$ stands for the functor forgetting the equivariance with respect to the diagonal embedding $G \to G^2$. 

Recall from \cite[Chapter 7]{etingofTensorCategories2016} that for a monoidal category $\cA$ and a category $\cM$ we have a notion of $\cM$ being a (left or right) module category over $\cA$. If $\cM, \mathcal{N}$ are two module categories over $\cA$, let $\operatorname{Fun}_{\cA}(\cM,\mathcal{N})$ stand for the category of $\cA$-module functors from $\cM$ to $\mathcal{N}$, see \cite[Definitions 7.2.1, 7.2.2]{etingofTensorCategories2016}. 

Write $\mathcal{Z}\Hecke{1}:=\operatorname{Fun}_{\Hecke{2}}(\Hecke{1},\Hecke{1})$ for the category of module-endofunctors of $\Hecke{1}$ over $\Hecke{2}$. Recall that, by definition, an object of $\mathcal{Z}\Hecke{1}$ is a functor $F:\Hecke{1}\to\Hecke{1}$ together with natural isomorphisms 
\[
  s_{\mathcal{B},\mathcal{A}}^F:\mathcal{B}\bowtie F(\mathcal{A}) \rsim F(\mathcal{B}\bowtie \mathcal{A}),
\]
satisfying certain compatibility conditions. A morphism between functors $F, G$ in $\mathcal{Z}\Hecke{1}$ is a natural transformation $\tau: F \to G$ such that the diagrams of the form
\begin{equation}
  \label{eq:9}
  \begin{tikzcd}
    \mathcal{B}\bowtie F(\mathcal{A}) \ar[d,"\mathcal{B}\, \bowtie\,\tau_\mathcal{A}"'] \ar[r, "s_{\mathcal{B},\mathcal{A}}^F"] & F(\mathcal{B} \bowtie \mathcal{A}) \ar[d,"\tau_{\mathcal{B} \bowtie \mathcal{A}}"]  \\
    \mathcal{B} \bowtie G(\mathcal{A}) \ar[r, "s^G_{\mathcal{B},\mathcal{A}}"']&   G(\mathcal{B} \bowtie \mathcal{A})    
  \end{tikzcd}
\end{equation}
are commutative. 

Let $G/_{\Ad}G$ stand for the quotient stack of $G$ with respect to the adjoint action. The category $D^b(G/_{\Ad}G)$ is monoidal with respect to the convolution operation defined by the following formula. Consider the diagram
\[
  \begin{tikzcd}
    & G\times G\arrow[ld,"p_{1}"']\arrow[rd,"p_{2}"]\arrow[r,"m"]  & G \\
G &                                      & G 
  \end{tikzcd}
\]
For $\mathcal{A},\mathcal{B} \in D^b(G/_{\Ad}G),$ we define
\[
  \mathcal{A}\star\mathcal{B} = m_{!}(p_1^*\mathcal{A} \otimes p_{2}^*(\mathcal{B})).
\]

Our main result is
\begin{theorem}
  \label{sec:hecke-categ-char}
 In all of the settings \ref{set:c}-\ref{set:mixed}, there is an equivalence of monoidal categories 
 \begin{equation}
   \label{eq:1}
   \tilde{a}:D^b(G/_{\Ad}G) \rsim \mathcal{Z}\Hecke{1}.
 \end{equation}
\end{theorem}
Note that a priori $\mathcal{Z}\Hecke{1}$ is not a triangulated category. Define the evaluation functor
\[
  \varepsilon: \mathcal{Z}\Hecke{1} \to \Hecke{1}, F \mapsto F(\Hu), 
\]
where $\Hu$ stands for the monoidal unit in $\Hecke{1}$.
During the course of the proof of Theorem ~\ref{sec:hecke-categ-char} we will deduce the following
\begin{corollary}
  We say that the triangle in $\mathcal{Z}\Hecke{1}$ is distinguished, if it becomes a distinguished triangle in $\Hecke{1}$ after the application of $\varepsilon$. The category $\mathcal{Z}\Hecke{1}$ equipped with the set of distinguished triangles defined in this way and an obvious shift functor $[1]$ is a triangulated category, and $\tilde{a}$ is an exact equivalence.
\end{corollary}
\begin{bremark}
  The category $\mathcal{Z}\Hecke{1}$ is not, a priori, braided monoidal. The equivariant derived category $D^b(G/_{\Ad}G)$, on the other hand, is canonically braided, see \cite{boyarchenkoCharacterSheavesUnipotent2013}. It would be interesting to see if the category $\mathcal{Z}\Hecke{1}$ can also be equipped with the braiding without reference to Theorem \ref{sec:hecke-categ-char}.

  Note that the equivalence involving the abelian Drinfeld center, proved below as Theorem \ref{sec:abel-categ-char}, is compatible with the braiding.
\end{bremark}
\section{Sheaves on a group as comodules over the Springer comonad}
\label{sec:sheaves-group-as}
\subsection{Comonads in triangulated categories.}\label{sec:comon-triang-categ-8}
All categories we are working with are assumed to be additive (in fact, $\mathbb{k}$-linear).

For a pair of functors $F:\mathcal{C}\rightleftarrows \mathcal{D}:G$ we write $F\dashv G$ to indicate that $F$ is left adjoint to $G$. 

For a comonad $S$ on a category $\mathcal{D}$, let $\comod{S}{\cD}$ stand for the category of $S$-comodules in $\mathcal{D}$. We have an adjoint pair $F_S\dashv G_S$
\[
  F_S:\comod{S}{\cD} \rightleftarrows\cD:G_S,
\]
such that $F_S\dashv G_S$ with $F_S$ being a functor forgetting the comodule structure. 

\begin{definition}
  A comonad $S$ is called separable, if the comultiplication map $\Delta: S \to SS$ admits a natural retraction \[s: SS\to S, \enskip s\Delta = \operatorname{id},\] as a map of $S$-bicomodules, that is such that the following diagram is commutative:
\[
  \begin{tikzcd}
      & SS \arrow[ld, "S\Delta"'] \arrow[rd,"\Delta S"] \arrow[d, "s"'] & \\
      SSS \arrow[rd, "sS"'] & S \arrow[d, "\Delta"'] & SSS \arrow[ld,"Ss"] \\
      & SS &
  \end{tikzcd}
\]
\end{definition}

Let $F:\mathcal{C}\rightleftarrows \mathcal{D}:G$ be a pair of functors, $F \dashv G$. We say that the unit transformation 
\(  \eta: \Id_{\mathcal{C}} \to GF\)
is naturally split if there exists a transformation $\sigma: GF \to \Id_{\mathcal{C}}$ with $\sigma\circ\eta = \operatorname{id}$.  

\begin{proposition}
 \label{sec:comon-triang-categ-6} 
Let $F:\mathcal{C}\rightleftarrows \mathcal{D}:G$ be a pair of functors, $F \dashv G$. 
 \begin{enumerate}[label=\alph*),ref=(\alph*)]
\item If the unit transformation  
\[
  \eta: \Id_{\mathcal{C}} \to GF
\]
is naturally split, then the comonad $S = FG$ is separable.

\item If the comonad $S$ is separable, then the unit transformation
\[
  \eta_S: \Id_{\comod{S}{\cD}} \to G_SF_S
\]
is naturally split.
 \end{enumerate}
\end{proposition}
\begin{proof}
  For the fact that $S$ is separable if and only if $\eta_S$ is naturally split see, for example, \cite[Proposition 3.11]{balmerSeparabilityTriangulatedCategories2011} and references therein. We have that $s:FGFG\to FG$ is given by the formula $s = F\sigma G$, where $\sigma$ is the splitting of $\eta$.
\end{proof}
\begin{bremark}
  Note that the separability of $S = FG$ does not imply the splitting of $\eta$. For example, for any separable $S$, $F$ and $G$ can be chosen such that $F$ is not faithful.  
\end{bremark}
\begin{proposition}
  \label{sec:co-monads-triang}
  Assume that $F:\mathcal{C}\rightleftarrows \mathcal{D}:G$ is a pair of exact functors, $F \dashv G$, between idempotent complete pre-triangulated categories $\mathcal{C}, \mathcal{D}$ such that the unit transformation 
  \(  \eta: \Id_{\mathcal{C}} \to GF\) is naturally split. Define a comonad $S = FG$. Then
  \begin{enumerate}[label=\alph*),ref=(\alph*)]
  \item\label{item:2} A triangle $X \to Y \to Z \to X[1]$  is distinguished in $\mathcal{C}$ if and only if a triangle $FX \to FY \to FZ \to FX[1]$ is distinguished in $\mathcal{D}$.
  \item \label{item:1} The natural functor $\mathcal{C} \to S-\operatorname{comod}_{\mathcal{D}}$, induced by the functor $F$, is an equivalence of categories. 
  \end{enumerate}
\end{proposition}
To prove this we will need the following
\begin{proposition}
  \label{sec:comon-triang-categ}
  Let $F:\mathcal{C}\rightleftarrows \mathcal{D}:G$ be a pair of exact functors, $F \dashv G$, between \emph{pre-triangulated} categories $\mathcal{C}, \mathcal{D}$. Assume that $F$ is a faithful functor. Then any object of $X \in \operatorname{\cC}$ is a direct summand of the object $GF X$.
\end{proposition}
\begin{proof}
  (cf. \cite[Proposition 2.10]{balmerSeparabilityTriangulatedCategories2011}). Consider the completion of the unit map $\eta_X: X \to GFX$ to a distinguished triangle \[X \xrightarrow{\eta_X} GFX \to C \xrightarrow{\sigma} X[1].\] 
  By definition, the composition
  \[
    FX \xrightarrow{F\eta_X} FGF X \xrightarrow{\theta_{FX}} FX  
  \]
  is the identity, where $\theta:FG \to \Id_{\cD}$ is the counit transformation, so that the map $F\eta_X$ is split and $F(\sigma) = 0$. Since $F$ is faithful, $\sigma = 0$ and so $\eta_X$ is split.
\end{proof}

We will also use the following statements.
\begin{theorem}[\cite{balmerSeparabilityTriangulatedCategories2011}]
  \label{sec:comon-triang-categ-1}
Let $\cD$ be a pre-triangulated category and let $\cC$ be an idempotent-complete suspended category. Let \[F : \cC \rightleftarrows \cD : G\] be a pair of functors commuting with suspension, $F \dashv G$. Assume that the unit of $GF$ is naturally split, and that the comonad $S = FG$ on $\cD$ is exact. Then $\cC$ is pre-triangulated with
distinguished triangles $\Delta$ being exactly the ones such that $F(\Delta)$ is distinguished in $\cD$. Moreover,
with this pre-triangulation both functors $F$ and $G$ become exact.
\end{theorem}
\begin{proof} 
This is \cite[Theorem 4.1]{balmerSeparabilityTriangulatedCategories2011}, stated for comonads instead of monads.
\end{proof}
\begin{corollary}[\cite{balmerSeparabilityTriangulatedCategories2011}]
 \label{sec:comon-triang-categ-2} 
 Let $\cD$ be an idempotent-complete pre-triangulated category and let $S:\cD\to\cD$ be an exact separable comonad. Then the category $\comod{S}{\cD}$ is pre-triangulated, so that the functors $F_S, G_S$ are exact.  The pre-triangulation is characterized by exactness of either $F_S$ or $G_S$.   
\end{corollary}
\begin{proof} 
This is \cite[Corollary 4.3]{balmerSeparabilityTriangulatedCategories2011}, stated for comonads instead of monads.
\end{proof}
\begin{proof}[Proof of Proposition \ref{sec:co-monads-triang}]

  ``Only if'' statement of \ref{item:2} is just a consequence of $F$ being exact, so we only need to prove the other direction. Consider any triangle  \[X \to Y \to Z \to X[1]\]  such that the triangle \[FX \to FY \to FZ \to FX[1]\] is distinguished. Since the unit map $\eta$ is split, we conclude that $X \to Y \to Z \to X[1]$ is a direct summand of the distinguished triangle \[GFX \to GFY \to GFZ \to GFX[1],\] and hence is distinguished itself.

  We now prove \ref{item:1}. For (an arbitrary) comonad $S$ on a category $\mathcal{D}$, let $\cofree{S}{\cD}$ stand for its cofree category. Recall that this category has the same objects as $\mathcal{D}$, and
  \[
    \Hom_{\cofree{S}{\cD}}(A,B) := \Hom_{\cD}(SA,B),
  \]
  with the composition
   \[\Hom(B,C)_{\cofree{S}{\cD}}\times\Hom_{\cofree{S}{\cD}}(A,B) \to \Hom_{\cofree{S}{\cD}}(A,C)\]
    given by the composition of maps 
  \[
    (f,g) \mapsto SA \to SSA \xrightarrow{S(f)} SB \xrightarrow{g} C,
  \]
  with the first map coming from the co-multiplication.

  For $S = FG$, functor $G$ defines a fully faithful functor \[G_{\cf}:\cofree{S}{\cD} \to \cC.\] Since the unit transformation $\eta$ is split, $F$ is faithful, so by Proposition \ref{sec:comon-triang-categ}, the fully faithful functor
  \[
   G_{\cf}^{\natural}:(\cofree{S}{\cD})^{\natural} \to \cC^{\natural} \simeq \cC,
 \]
 where $(-)^{\natural}$ stands for the idempotent completion of a category, is essentially surjective, and so is an equivalence.

By Proposition \ref{sec:comon-triang-categ-6}, the comonad $S$ is separable, so by Corollary~\ref{sec:comon-triang-categ-2}, the category $\comod{S}{\cD}$ is pre-triangulated and functors $F_S,G_S$ are exact. The functor $F_S$ is always faithful, and the category $\comod{S}{\cD}$ is idempotent-complete, bacause $\cD$ is such, so by the same argument as above we have an equivalence \[(\cofree{S}{\cD})^{\natural} \simeq \comod{S}{\cD},\]
induced by the functor $G_{S,\cf}:\cofree{S}{\cD}\to\comod{S}{\cD}.$

Let $\bar{F}:\cC\to\comod{S}{\cD}$ be the canonical functor induced by $F$. One has $\bar{F}\circ G_{\cf} \simeq G_{S,\cf}$, with $G_{\cf}^{\natural},G_{S,\cf}^{\natural}$ being equivalences. It follows that $\bar{F}^{\natural}\simeq \bar{F}$ is an equivalence as well.

\end{proof}
\begin{bremark}
  Note that the category $S-\operatorname{comod}_{\mathcal{D}}$ has no reason to be (pre-)triangulated for an arbitrary triangulated comonad $S$ on a (pre-)triangulated category $\mathcal{D}$. It follows that, in the assumption of the Proposition \ref{sec:co-monads-triang}, we can equip $S-\operatorname{comod}_{\mathcal{D}}$ with a triangulated structure, saying that the triangle in the latter category is distinguished if its image under the forgetful functor to $\mathcal{D}$ is distinguished. 
\end{bremark}
\begin{bremark}
  The proof of Proposition \ref{sec:co-monads-triang} mostly follows \cite{balmerSeparabilityTriangulatedCategories2011} (see also \cite{dellambrogioNoteTriangulatedMonads2018}). More specifically, it is proved in \cite{balmerSeparabilityTriangulatedCategories2011} that for any stably-separable exact (co-)monad $M$ on a pre-triangulated (respectively, an $N$-triangulated, $N \geq 2$) category, the category of \mbox{(co-)}modules over it is pre-triangulated (respectively, $N$-triangulated) so that an \mbox{($N$-)}traingle is distinguished in $M-\operatorname{(co-)mod}_{\mathcal{D}}$ if and only if it is distinguished after the application of the forgetful functor to $\mathcal{D}$. In \cite{dellambrogioNoteTriangulatedMonads2018} it is proved that, for a monad $M$ on $\mathcal{C}$, if the forgetful functor from a triangulated category $M-\operatorname{mod}_{\mathcal{C}}$ to $\mathcal{C}$ is exact, then, for any triangulated realization $F:\mathcal{C}\rightleftarrows \mathcal{D}:G$ of a monad $M$ with $\mathcal{D}$ idempotent-complete and such that $G$ is conservative, we have a triangulated equivalence $\mathcal{D}\simeq M-\operatorname{mod}_{\mathcal{C}}$. We refer to the citation above for the definitions of the terms used in this remark. 
\end{bremark}
\subsection{Springer comonad.} We now apply the results of Subsection \ref{sec:comon-triang-categ-8} to our geometric setup. 

Consider the diagram
\[
\begin{tikzcd} & G\times G/B \arrow[ld, "p"'] \arrow[rd, "q"] & \\ G &
& \yy,
\end{tikzcd}
\]
where $p$ is the projection and $q$ is the quotient of the map \[q':G \times G/U \to G/U \times G/U, q'(g,xU) = (xU,gxU)\] by the free right $T$-action (with respect to which it is easy to see that $q'$ is equivariant). Consider the adjoint action of $G$ on itself and left diagonal action of $G$ on the other corners of the diagram above. This makes the diagram $G$-equivariant. The functor
\[
  \hc=q_!p^*:D^b(G/_{\Ad}G) \to D^b(G\backslash\yy)
\]
is called the Harish-Chandra transform. It is well-known (see, for example, \cite{ginzburgAdmissibleModulesSymmetric1989}) to be monoidal with respect to the convolutions $\star$ and $\star_1$.
It has a right adjoint functor
\[
  \chi = p_*q^!.
\]
We have a ``Springer comonad'' $\cS = \hc\circ\chi$ on $\Hecke{1}$. The terminology is justified by the following discussion.

Write $x \mapsto {}^gx := gxg^{-1}$ and define
\[
  \tilde{\mathcal{N}} := \{(x, gB)\in G\times G/B: x \in {}^gU\}.
\]
Let $\pi: \tilde{\mathcal{N}} \to G$ be the natural projection, and let \[\Spgr = \pi_*\cc_{\tilde{\mathcal{N}}}[2 \dim U]\] 
be the Springer sheaf.
We will use the following well-known
\begin{proposition}
\label{sec:springer-comonad-2}
The sheaf $\delta_e:=\iota_{e*}\cc_{\operatorname{pt}}$, where \mbox{$\iota_e:\operatorname{pt}\to G$} is the unit map of $G$, is a direct summand of $\Spgr$. More precisely, there is a $W$-action on $\Spgr$, such that $\Spgr^W\simeq\delta_e$. \end{proposition}
\begin{proof}
 In all settings except \ref{set:p} this is classical, see \cite{borhoRepresentationsGroupesWeyl1981}. In the setting \ref{set:p}, this can be deduced, for example, from \cite{acharWeylGroupActions2014}.
\end{proof}
We now recall the following results, proved in \cite{ginzburgAdmissibleModulesSymmetric1989} in the setting \ref{set:dmod}, with proofs transported verbatim to other settings.

\begin{lemma}\cite[Lemma 8.5.4]{ginzburgAdmissibleModulesSymmetric1989}
  \label{sec:springer-comonad-4}
There is a natural isomorphism  
\[\chi\circ\hc(-) \simeq \Spgr\star(-).\] 
\end{lemma}
\begin{proof}
Follows from base change isomorphism and diagram chase.
\end{proof}
\begin{proposition}
  \label{sec:springer-comonad-1}
(\cite[Theorem 8.5.1]{ginzburgAdmissibleModulesSymmetric1989})  The unit transformation \mbox{$\eta:\Id_{D^b(G/_{\Ad}G)}\to\chi\circ\hc$} for the comonad $\cS$ is naturally split.
\end{proposition}
\begin{proof}
  Apply Lemma \ref{sec:springer-comonad-4} and note that $\delta_e$ is the monoidal unit in $D^b(G/_{\Ad}G)$. By Proposition \ref{sec:springer-comonad-2} we get the result. 
\end{proof}
Since $D^b(G/_{\Ad}G)$ and $\Hecke{1}$ are idempotent complete, we now have the following consequence of Propositions \ref{sec:springer-comonad-1} and \ref{sec:co-monads-triang}.

\begin{corollary}
  \label{sec:springer-comonad}
 The Harish-Chandra functor $\hc$ induces an equivalence of categories \[\mathfrak{F}:D^b(G/_{\Ad}G)\simeq\comod{\cS}{\Hecke{1}}.\] 
\end{corollary}
\subsection{Springer comonad and bimodule structure.}
We record the following simple observation. 

\begin{proposition}
  \label{sec:comon-module-categ-1}
Let $\cA$ be a monoidal category acting on a category $\cM$, and $\bar{S}$ be a coalgebra object in $\cA$. Action of $\bar{S}$ defines a comonad on $\cM$, denoted by $S$. Any functor $F \in \operatorname{Fun}_{\cA}(\cM,\cM)$ sends $S$-comodules to $S$-comodules.
\end{proposition}

We now show that the Springer comonad fits into the setting of Proposition \ref{sec:comon-module-categ-1}.
\begin{proposition}
  \label{sec:spring-comon-bimod}
 There is a coalgebra object $\bar{\cS} \in \Hecke{2}$ such that there is an isomorphism of comonads
  \begin{equation}
    \label{eq:3}
    \bar{\cS}\bowtie(-)\simeq\cS(-).
  \end{equation}
\end{proposition}
\begin{proof}
  
Write
\[
  \alpha: G \times G/B \times G/B \to \yy^{(2)}, \alpha(g,xB,yB) = (xU,yU,gyU,gxU)
\]
(here and below, by the latter we mean the corresponding class in $\yy^{(2)}$) and define the action of $G^2$ on $G \times G/B \times G/B$ using the formula
\[
  (h_1,h_2)\cdot(g,xB,yB) = (h_2gh_1^{-1}, h_1xB,h_1yB),
\]
so that $\alpha$ is a $G^2$-equivariant map.
Let \[\bar{\cS}:= \alpha_!\cc_{(G\times G/B\times G/B)/G^2}[2\dim U] \in \Hecke{2}.\]

From the projection formula and proper base change, it is easy to see that that we have an isomorphism of functors as in \eqref{eq:3}.

To show that $\bar{\cS}$ is a coalgebra object, consider the diagram
\[
\begin{tikzcd}
          & G\times G/B \times \yy \arrow[ld, "p'"'] \arrow[rd, "q'"] &           \\
G\times \yy &                                                       & \yy^{(2)},
\end{tikzcd}
\]
where the maps are given by
\[
  p'(g,xB,x'U,y'U) = (g, x'U,y'U),
\]
\[
  q'(g,xB,x'U,y'U) = (xU,x'U,y'U,gxU),
\]
which are well defined since they are compatible with the right actions of $T$ and $T^2$ on $Y^2$ and $Y^4$.
We define an action of $G^2$ on \(G \times G/B \times \yy\) using the formula
\[
  (h_1,h_2)\cdot(g,xB,x'U, y'U) = (h_2gh_1^{-1}, h_1xB,h_1x'U,h_2y'U)
\]
and on $G\times \yy$ in an evident way to make the map $p'$ equivariant.

We now observe that the map $p'$ is proper, the map $q'$ is smooth, so that the functor $q'_!p'^*$ is left adjoint to the functor $p'_*q'^! \simeq p'_!q'^*[2\dim U]$. Moreover, we have an isomorphism of functors
\begin{equation}
\label{eq:4}
  \bar{\cS}\star_2(-)\simeq q'_!p'^*p'_!q'^*[2 \dim U],
\end{equation}
so that convolution with $\bar{\cS}$ is indeed a comonad, and $\bar{\cS}$ is a coalgebra object. Diagram chase shows that the isomorphism $\bar{\cS}\bowtie(-)\simeq \cS(-)$ is compatible with the comultiplication. 
\end{proof}
\section{Proof of the main Theorem.}
\label{sec:dbgg-as-categorical}
\subsection{Centralizer category as a category of comodules.}  
Recall the notation
\[
 \cZ\Hecke{1} := \operatorname{Fun}_{\Hecke{2}}(\Hecke{1},\Hecke{1})
\]
for the category of module-endofunctors of $\Hecke{1}$ over $\Hecke{2}$.

Recall the evaluation functor \[\varepsilon:\cZ\Hecke{1} \to \Hecke{1}, F \mapsto F(\Hu),\]
and the equivalence \[\mathfrak{F}:D^b(G/_{\Ad} G) \to \comod{\cS}{\Hecke{1}}\] from Corollary \ref{sec:springer-comonad}.

We will deduce Theorem~\ref{sec:hecke-categ-char} from the existence of the following diagram, satisfying properties stated in Propositions \ref{sec:sheaves-group-as-1}, \ref{sec:centr-categ-as-1}:
\begin{equation}
\label{eq:2} 
\begin{tikzcd}
\comod{\cS}{\Hecke{1}}                                                         & \cZ\Hecke{1}\arrow[d,"\varepsilon"] \arrow[l, "\tilde{b}"', bend right] \\
D^b(G/_{\Ad}G) \arrow[ru, "\tilde{a}"] \arrow[r, "\hc"'] \arrow[u, "\mathfrak{F}", "\sim" {rotate=90, anchor=north}] & \Hecke{1}. \arrow[l, "\chi"', shift right]      
\end{tikzcd}
\end{equation}
\begin{proposition}
  \label{sec:sheaves-group-as-1}
  \begin{enumerate}[label=\alph*),ref=(\alph*)]
  \item\label{item:3} There is a monoidal functor \[\tilde{a}:D^b(G/_{\Ad}G)\to\cZ\Hecke{1},\] 
    satisfying $\varepsilon \circ \tilde{a} \simeq \hc$.
  \item\label{item:4} There is a functor \[\tilde{b}:\cZ\Hecke{1}\to\comod{\cS}{\Hecke{1}},\]
satisfying $\hc\circ\mathfrak{F}^{-1}\circ\tilde{b} \simeq \varepsilon$.
  \item\label{item:5} Functors $\tilde{a},\tilde{b}$ satisfy
    \[
      \tilde{b}\circ\tilde{a} \simeq \mathfrak{F}.
    \]
  \item \label{item:11} The functor $\tilde{a}$ is fully faithful. 
  \end{enumerate}
\end{proposition}
\begin{proposition}
  \label{sec:centr-categ-as-1}
  There is a natural isomorphism of functors
  \[
    \tilde{a}\circ\chi\circ\varepsilon(-)\simeq (-) \circ \tilde{a}(\Sigma),
  \]
  where $\circ$ denotes the monoidal structure on $\cZ\Hecke{1}$ coming from the composition of functors.
\end{proposition}
\begin{proof}[Proof of Theorem \ref{sec:hecke-categ-char}.]
We claim that the monoidal functor $\tilde{a}$ of Proposition \ref{sec:sheaves-group-as-1} \ref{item:3} is an equivance of monoidal categories. For this it suffices to show that $\tilde{a}$ is an equivalence of plain categories.

By Propositions \ref{sec:centr-categ-as-1} and \ref{sec:springer-comonad-2} every object in $\mathcal{Z}\Hecke{1}$ is a direct summand of an object in the image of $\tilde{a}$. Since both source and target of $\tilde{a}$ are idempotent complete, and $\tilde{a}$ is fully faithful by Proposition \ref{sec:sheaves-group-as-1} \ref{item:11}, we get that $\tilde{a}$ is essentially surjective, hence the result.
\end{proof}
\begin{corollary}
The functor $\chi\circ\varepsilon$ is equipped with $W$-action, such that we can express the functor inverse to $\tilde{a}$ as $(\chi\circ\varepsilon)^W$, where $(-)^{W}$ stands for the functor of taking invariants with respect to $W$.
\end{corollary}
\begin{proof}
Note that by Proposition \ref{sec:sheaves-group-as-1}~\ref{item:4} \ref{item:5}, the functor $\chi\circ\varepsilon$ is isomorphic to the functor $\chi \circ \hc \circ \mathfrak{F}^{-1} \circ \tilde{b}$, and so is equipped with $W$-action, coming from the action of $W$ on $\chi\circ\hc \simeq \Sigma \star (-)$ (see Proposition \ref{sec:springer-comonad-2} and  Lemma \ref{sec:springer-comonad-4}).
\end{proof}
\subsection{Proof of Proposition \ref{sec:sheaves-group-as-1}.}
We will need the following
\begin{lemma}
  \label{sec:centr-categ-as}
There are monoidal functors $L,R:\Hecke{1}\to\Hecke{2}$ such that 
  there are isomorphisms
    \[
      \mathcal{A} \star_1 \mathcal{B} \simeq L(\mathcal{A})\bowtie \mathcal{B},
    \]
    \[
        \mathcal{B} \star_1 \mathcal{A} \simeq R(\mathcal{A})\bowtie  \mathcal{B},
    \]
    functorial in $\mathcal{A},  \mathcal{B} \in \Hecke{1}$.
  \end{lemma} 
  \begin{proof}
We construct the functor $L$, the functor $R$ is constructed completely analogously. Consider the space $Y^3$ together with the action of $G^2$ on the left given by 
\[
(g_1,g_2)\cdot(x_1U, x_2U, x_3U)\mapsto(g_1x_1U,g_1x_2U, g_2x_3U)
\]
and the diagonal action of $T$ on the right. 
Then the closed embedding $Y^3\to Y^4$ given by
\[
(x_1U, x_2U, x_3U)\mapsto (x_1U, x_2U, x_3U, x_3U)
\]
induces a $G^2$-equivariant map $\iota\colon Y^3/T\to\yy^{(2)}$. There is also a projection on the first two coordinates $p_{12}\colon Y^3\to Y^2$,
which induces a map $p_{12}\colon Y^3/T\to\yy$. This map is $G^2$-equivariant, with the action of the first copy of $G$ on $\yy$ being the usual diagonal action and the action of the second copy being trivial. Abusing notation, write $p_{12}$ also for the composition \[p_{12}:G^2\backslash Y^3/T \to G^2\backslash\yy \to G\backslash \yy,\] where the second map is associated to the first projection $G^2 \to G.$ We define
\[
L:=\iota_*p_{12}^*:\Hecke{1} \to \Hecke{2}.
\]

To prove the required property, observe that the projection $p_{14}$ in the definition of the $\bowtie$ restricted to the image of $\iota$ 
is exactly the projection $p_{13}\colon Y^3/T\to\yy$ used in the definition of the $\star_1$. The sheaves being pushed forward are identified by $T$-equivariance. 

Finally, to observe the monoidality of the functor $L$ notice that we have 
\[p_{1256}^*L(\mathcal{A})\otimes p_{2345}^*L(\mathcal{B})\simeq\iota^{(2)}_*(p_{12}^*\mathcal{A}\otimes p_{23}^*\mathcal{B}),\] 
where $\iota^{(2)}\colon Y^4/T\to Y^6/T^3$ is induced by 
\[
(x_1U, x_2U, x_3U, x_4U)\mapsto (x_1U, x_2U, x_3U, x_4U, x_4U, x_4U)
\]
and $p_{ij}\colon Y^4/T\to\yy$ are the projections.
Moreover, the diagram 
\[
  \begin{tikzcd}
   Y^3/T  \ar[d, "\iota"'] & Y^4/T \ar[l,"p_{134}"']\ar[d,"\iota^{(2)}"] \\
    \yy^{(2)} & Y^6/T^3 \ar[l,"p_{1346}"]
  \end{tikzcd}
\]
is commutative implying that
\[
L(\mathcal{A})\star_2L(\mathcal{B})\simeq\iota_*p_{134!}(p_{12}^*\mathcal{A}\otimes p_{23}^*\mathcal{B}).
\]
It remains to see that the diagram
\[
  \begin{tikzcd}
   Y^3/T  \ar[d, "p_{12}"'] & Y^4/T \ar[l,"p_{134}"']\ar[d,"p_{123}"] \\
    \yy & Y^3/T \ar[l,"p_{13}"]
  \end{tikzcd}
\]
is Cartesian and base change will give 
\[
L(\mathcal{A})\star_2L(\mathcal{B})\simeq L(\mathcal{A}\star_1\mathcal{B})
\]
as desired.

\end{proof}
  \begin{proof}[Proof of Proposition \ref{sec:sheaves-group-as-1}.]
  We first prove \ref{item:3}.
  We define $\tilde{a}$ as 
\[
\tilde{a}\colon \F\mapsto -\star_1\hc(\F),
\]
together with the central structure provided by an argument similar to \cite[Proposition 9.2.1(ii)]{ginzburgAdmissibleModulesSymmetric1989}. In more details, we need to define the canonical isomorphism
\begin{equation}
\label{eq:15}
s^{\tilde{a}(\F)}_{\mathcal{B},\mathcal{A}}\colon \mathcal{B}\bowtie (\mathcal{A}\star_1\hc(\F))\xrightarrow{\sim}(\mathcal{B}\bowtie \mathcal{A})\star_1\hc(\F)
\end{equation}
for each $\mathcal{A}$ and $\mathcal{B}$. 
Combining diagrams used to define $-\star_1-$ and $\hc$ one sees that $-\star_1\hc(\F)$ is given by the push forward of the sheaf $\F\boxtimes -$ along the map
\[
f\colon G\times\yy\to\yy,\enskip (g,x_1U,x_2U)\mapsto(x_1U,gx_2U)
\] 
Therefore, we can write the right hand side of (\ref{eq:15}) as push forward along the map $f\circ(\mathrm{id}\times p_{14})\colon G\times\yy^{(2)}\to \yy$: 
\[(\mathcal{B}\bowtie \mathcal{A})\star_1\hc(\F)\simeq\left(f\circ(\mathrm{id}\times p_{14})\right)_*\left(\F\boxtimes(\operatorname{For}^{G^2}_G(\mathcal{B}) \otimes p_{23}^*(\mathcal{A}))\right).\] On the other hand, to compute the left hand side of (\ref{eq:15}) by the base change and the projection formula we take the push forward along the map $p_{14}\colon G\times\yy^{(2)}\to\yy$: \[\mathcal{B}\bowtie (\mathcal{A}\star_1\hc(\F))\simeq p_{14,*}\left(pr_G^*\F\otimes \operatorname{For}^{G^2}_G(f_3^*\mathrm{inv}_G^*\mathcal{B})\otimes p_{23}^*(\mathcal{A})\right),\] where $pr_G$ is the projection on the $G$ factor and $f_3$ is the map of $G$-action on the third coordinate and $\mathrm{inv}_G$ is the inversion map $g\mapsto g^{-1}$ on $G$ factor. Using the $G$-equivariance of $\mathcal{B}$ with respect to the second copy of $G$ we get the canonical isomorphism $f_3^*\mathrm{inv}_G^*\mathcal{B}\simeq f_4^*\mathcal{B}$, where $f_4$  is the map of $G$-action on the fourth coordinate. We finally get the desired isomorphism by the projection formula and the identification $f\circ(\mathrm{id}\times p_{14})=p_{14}\circ f_4$.

We now prove \ref{item:4}.  First recall that, by Proposition~\ref{sec:spring-comon-bimod} there is a coalgebra object $\bar{\cS} \in \Hecke{2}$ such that there is an isomorphism of comonads
\begin{equation}
\bar{\cS}\bowtie(-)\simeq\cS(-).
\end{equation}
Define $\tilde{b}(F) := F(\Hu)$, an $\cS$-comodule by Proposition~\ref{sec:comon-module-categ-1} together with the observation that $\Hu \simeq \hc(\delta_e)$ has a canonical $\cS$-comodule structure as an image of an object under $\hc$, by the general discussion of Subsection \ref{sec:comon-triang-categ-8}. 

The isomorphism $\hc\circ\mathfrak{F}^{-1}\circ\tilde{b} \simeq \varepsilon$ is again immediate from the definition.
 This completes the proof of \ref{item:4}.

For \ref{item:5}, note that we have the following isomorphism of plain objects in $\Hecke{1}:$ 
\[
\tilde{b}\circ\tilde{a}(\F)\simeq \hc(\F)\star_1\hc(\delta_e)\simeq \hc(\F),
\]
where the last equivalence is by monoidality of $\hc$. The $\cS$-comodule on $\hc(\F)$ comes from the $\cS$-comodule structure on the second factor $\hc(\delta_e)$. 
In other words, it is given by the composition 
\begin{multline*}
 \hc(\F)\simeq \hc(\F)\star_1\hc(\delta_e)\to\\
  \to\hc(\F)\star_1\cS(\hc(\delta_e))\simeq \hc(\F)\star_1\hc(\Sigma)\simeq \cS(\hc(\F)).
\end{multline*}
The functor $\mathfrak{F}$ is essentially defined by the same map, i.e. ${\hc(\delta_e)\to\hc(\Sigma)}$ convolved with the identity map on $\hc(\F)$.  

To prove \ref{item:11} we first note that by \ref{item:5} we have $\hc \simeq \varepsilon\circ\tilde{a}$, and $\hc$ is a faithful functor by Proposition \ref{sec:springer-comonad-1}, so $\tilde{a}$ is also faithful.

We now show that the functor $\varepsilon$ is faithful. Recall that a morphism between $F, G \in \mathcal{Z}\Hecke{1}$ is a natural transformation $\tau:F \to G$ making diagrams of the form \eqref{eq:9} commutative. It follows from Lemma \ref{sec:centr-categ-as} and diagram
\[
  \begin{tikzcd}
    F(\mathcal{A}) \simeq F(L(\mathcal{A})\bowtie\Hu)  \ar[d, "\tau_\mathcal{A}"'] & L(\mathcal{A})\bowtie F(\Hu) \ar[l,"s^F_{L(\mathcal{A}),\Hu}"']\ar[d,"L(\mathcal{A})\,\bowtie\,\tau_{\Hu}"] \\
    G(\mathcal{A}) \simeq G(L(\mathcal{A})\bowtie\Hu) & L(\mathcal{A})\bowtie G(\Hu) \ar[l,"s^G_{L(\mathcal{A}),\Hu}"]
  \end{tikzcd}
\]
being commutative, that such a transformation is defined by its value $\tau_{\Hu}$ on $\Hu$, and so $\varepsilon$ is indeed faithful. 

By \ref{item:5}, $\hc \circ \mathfrak{F}^{-1} \circ \tilde{b} \simeq \varepsilon$, and so $\tilde{b}$ is also faithful. Since $\tilde{b}\circ\tilde{a}$ is an equivalence by \ref{item:5}, it follows that $\tilde{a}$ is full, which finishes the proof of \ref{item:11}.
\end{proof}

\subsection{Proof of Proposition \ref{sec:centr-categ-as-1}.}  

We will need the following
\begin{lemma}
  \label{sec:proof-prop-refs}
  For any $\mathcal{A} \in \Hecke{1}, \mathcal{B} \in \Hecke{2}$ there is an isomorphism natural in $\mathcal{A}$ and $\mathcal{B}$
    \[
      c_{\mathcal{B},\mathcal{A}}:\mathcal{B} \star_2 L(\mathcal{A}) \star_2 \bar{\mathcal{S}} \rsim L(\mathcal{B} \bowtie \mathcal{A}) \star_2 \bar{\mathcal{S}},
    \]
    such that, for any $\F \in \Hecke{1}$, the diagram
    \[
      \begin{tikzcd}
        (\mathcal{B} \star_2 L(\mathcal{A}) \star_2 \bar{\mathcal{S}}) \bowtie \F \arrow[r]\arrow[d, "c_{\mathcal{B},\mathcal{A}}\bowtie\,\F"'] & \mathcal{B} \bowtie (\mathcal{A} \star_1 \mathcal{S}\F) \arrow[d,"s_{\mathcal{B},\mathcal{A}}^{\tilde{a}(\chi(\F))}"]\\
        (L(\mathcal{B}\bowtie \mathcal{A})\star_2\bar{\mathcal{S}})\bowtie \F \arrow[r] & (\mathcal{B} \bowtie \mathcal{A})\star_1 \mathcal{S}\F
      \end{tikzcd}
    \]
    is commutative. In the diagram, horisontal arrows are defined using the isomorphisms from Lemma \ref{sec:centr-categ-as} and $\bar{\mathcal{S}}\bowtie\F \simeq \mathcal{S}\F$.
\end{lemma}
\begin{proof}
We freely use the notations from the proof of Lemma \ref{sec:centr-categ-as}.

It is sufficient to verify the required condition for $\mathcal{A}=\Hu$. Indeed, if this is done we can set $\mathcal{B}':=\mathcal{B}\star_2 L(\mathcal{A})$, so that 
$\mathcal{B} \star_2 L(\mathcal{A}) \star_2 \bar{\mathcal{S}}=\mathcal{B}'\star_2\bar{\mathcal{S}}$ and $\mathcal{B} \bowtie \mathcal{A}=\mathcal{B}'\bowtie\Hu$ and we can set $c_{\mathcal{B},\mathcal{A}}:=c_{\mathcal{B}',\Hu}$ as long as we know the latter (note that $L(\Hu)$ is the monoidal unit of $\Hecke{2}$ as $L$ is the monoidal functor).

We now have to construct an isomorphism
\[
c_{\mathcal{B},\Hu}\colon \mathcal{B}\star_2\bar{\mathcal{S}}\to L(\mathcal{B}\bowtie\Hu)\star_2\bar{\mathcal{S}},
\]
satisfying the required compatibility. We have a description of the functor $-\star_2\bar{\mathcal{S}}$ similar to  (\ref{eq:4}), with $q'$ replaced by \[q'(g,xB,x'U,y'U) = (x'U,xU,gxU,y'U).\] It is, therefore, sufficient to construct a canonical isomorphism 
\[
p'_!q'^{*}\mathcal{B}\simeq p'_!q'^{*}(L(\mathcal{B}\bowtie\Hu)).
\]
By the definition, we can further unwrap the right hand side as 
\[
p'_!q'^{*}(L(\mathcal{B}\bowtie\Hu))\simeq p'_!q'^*\iota_*p_{12}^*p_{13,!}\Delta^*_{23}\mathcal{B},
\]
where $\Delta_{23}$ is the map induced by $Y^3\to Y^4$ sending the middle copy to the  diagonal between the second and the third copies. 
The following square is Cartesian:
\[
  \begin{tikzcd}
   G\times G/B\times\yy  \ar[d, "q'"'] & G\times\yy \ar[l,"\iota'"']\ar[d,"q''"] \\
    \yy^{(2)} & Y^3/T \ar[l,"\iota"],
  \end{tikzcd}
\]
where
\[
\iota'(g,x'U,y'U)=(g, g^{-1}y'B, x'U, y'U)
\]
and 
\[
q''(g,x'U,y'U)=(g^{-1}y'U, x'U, y'U).
\]
Indeed, restricting $q'$ to the image of $\iota$, we see that the Borel subgroup in the preimage is defined uniquely by the factors in $G\times\yy$ as $g^{-1}y'B$. Moreover, the composition $p'\circ\iota'$ is the identity map, which allows us to further rewrite
\[
p'_!q'^*\iota_*p_{12}^*p_{13,!}\Delta^*_{23}\mathcal{B}\simeq (p_{12}\circ q'')^*p_{13,!}\Delta^*_{23}\mathcal{B}.
\]
Note that 
\[
p_{12}\circ q''(g,x'U,y'U)=(g^{-1}y'U, x'U)
\]
and the following diagram is also Cartesian:
\[
  \begin{tikzcd}
    G\times\yy \ar[d,"p_{12}\circ q''"'] & G\times G/B\times\yy  \ar[d, "(p_{12}\circ q'')\times\mathrm{id}_{G/B}"] \ar[l,"p'"'] \\
    \yy & \yy\times G/B\ar[l,"p_{13}"'] .
  \end{tikzcd}
\]
Here we identify the image of the map $\Delta_{23}:Y^3/T^2 \to \yy^{(2)}$ with $\yy \times G/B$. Now we can further rewrite
\[
(p_{12}\circ q'')^*p_{13,!}\Delta^*_{23}\mathcal{B}\simeq p'_!(\Delta_{23}\circ((p_{12}\circ q'')\times\mathrm{id}_{G/B}))^*\mathcal{B}.
\]
It remains to notice that the composition $\Delta_{23}\circ((p_{12}\circ q'')\times\mathrm{id}_{G/B})$ equals to the composition of $q'\times\mathrm{id}_G\colon $ and the $G$-action map on the last two coordinates of $\yy^{(2)}$. The $G^2$-equivariance of $\mathcal{B}$ then provides the desired isomorphism. Moreover, the resulting composition is identified with the one from the constructed in Lemma \ref{sec:centr-categ-as}, proving the compatibility with the central structure. Note that both maps are coming from $G$-equivariance of $\mathcal{B}$ with respect to the action on the last two coordinates.
\end{proof}
\begin{proof}[Proof of Proposition \ref{sec:centr-categ-as-1}.]
First note that \[
\tilde{a}\circ{\chi}\circ\varepsilon(F)(\mathcal{A}) \simeq \tilde{a}(\chi(F(\Hu)) = \mathcal{A} \star_1 \mathcal{S}F(\Hu).
\]
We define a natural transformation
\begin{equation}
  \label{eq:13}
  \tau:\tilde{a}\circ{\chi}\circ\varepsilon(F) \to F \circ \tilde{a}(\Sigma)
\end{equation}
whose value on an object $\cA$ is the compostion $\tau_\mathcal{A}$ of isomorphisms
\[
  \mathcal{A} \star_1 \mathcal{S}F(\Hu) \xrightarrow{s^F_{\mathcal{S},\Hu}} \mathcal{A}\star_1 F(\hc(\Sigma)) \xrightarrow{s^F_{L(\mathcal{A}),\hc(\Sigma)}} F(\mathcal{A} \star_1 \hc(\Sigma)),
\]
where $s^F$ is the structure of an object in $\mathcal{Z}\Hecke{1}$ on $F$. To unburden the notations, let us denote left and right hand side of \eqref{eq:13} by $\Phi$ and $\Psi$, respectively. To show that $\tau: \Phi \to \Psi$ is an isomorphism of functors in $\mathcal{Z}\Hecke{1}$, we need to show that the diagram
\begin{equation}
  \label{eq:12}
  \begin{tikzcd}
    \mathcal{B}\bowtie(\mathcal{A}\star_1 \mathcal{S}F(\Hu)) \arrow[r,"s_{\mathcal{B},\mathcal{A}}^{\Phi}"] \arrow[d,"\mathcal{B}\,\bowtie\,\tau_\mathcal{A}"']                      & (\mathcal{B}\bowtie \mathcal{A})\star_1\mathcal{S}F(\Hu) \arrow[d,"\tau_{\mathcal{B}\bowtie \mathcal{A}}"] \\
\mathcal{B}\bowtie(F(\mathcal{A}\star_1\hc(\Sigma))) \arrow[r,"s_{\mathcal{B},\mathcal{A}}^{\Psi}"']     &   F((\mathcal{B}\bowtie \mathcal{A})\star_1\hc(\Sigma))                                 
  \end{tikzcd}
\end{equation}
is commutative for all $\mathcal{B} \in \Hecke{2}$. To do this, write
\[
\begin{tikzcd}
\mathcal{B}\bowtie(\mathcal{A}\star_1 \mathcal{S}F(\Hu)) \arrow[r, "s_{\mathcal{B},\mathcal{A}}^{\Phi}"] \arrow[d] \arrow[rd, phantom, "I" description]                      & (\mathcal{B}\bowtie \mathcal{A})\star_1\mathcal{S}F(\Hu) \arrow[d]                   \\
(\mathcal{B}\star_2 L(\mathcal{A})\star_2\bar{\mathcal{S}})\bowtie F(\Hu) \arrow[d,"s^F_{\mathcal{B}\star_2 L(\mathcal{A})\star_2\bar{\mathcal{S}},\Hu}"'] \arrow[r,"c"] \arrow[rd, phantom, "II" description]    & (L(\mathcal{B}\bowtie \mathcal{A})\star_2 \bar{\mathcal{S}})\bowtie F(\Hu) \arrow[d,"s^F_{L(\mathcal{B}\bowtie \mathcal{A})\star_2\bar{\mathcal{S}},\Hu}"] \\
F((\mathcal{B}\star_2 L(\mathcal{A}) \star_2 \bar{\mathcal{S}})\bowtie\Hu) \arrow[r, "c'"] \arrow[d] \arrow[rd, phantom, "III" description] & F((L(\mathcal{B}\bowtie \mathcal{A})\star_2 \bar{\mathcal{S}})\bowtie\Hu) \arrow[d] \\
F(\mathcal{B}\bowtie(\mathcal{A}\star_1 \hc(\Sigma))) \arrow[r,"F(s_{\mathcal{B},\mathcal{A}}^{\tilde{a}(\Sigma)})"'] \arrow[d, "(s_{\mathcal{B},\mathcal{A}\star_1\hc(\Sigma)}^{F})^{-1}"']                                                     & F((\mathcal{B}\bowtie \mathcal{A})\star_1\hc(\Sigma))                              \\
\mathcal{B}\bowtie(F(\mathcal{A}\star_1\hc(\Sigma))) \arrow[ru, bend right, "s_{\mathcal{B},\mathcal{A}}^{\Psi}"'] \arrow[ru, phantom, "IV" description, shift right=2]      &                                                               
\end{tikzcd}
\]
Here the unlabeled vertical morphisms are defined using the isomorphisms from Lemma \ref{sec:centr-categ-as} and $\bar{\mathcal{S}}\bowtie\F \simeq \mathcal{S}\F$. The morphism $c,c'$ are defined from the morphism of Lemma \ref{sec:proof-prop-refs} as follows:
\[
  c = c_{\mathcal{B},\mathcal{A}}\bowtie F(\Hu), c' = F(c_{\mathcal{B},\mathcal{A}}\bowtie \Hu).
\]
Diagram chase using the definition of module endofunctors shows that the compositions of vertical arrows are vertical arrows in \eqref{eq:12}. Diagrams $I$ and $III$ are commutative by Lemma \ref{sec:proof-prop-refs}, and diagrams $II$ and $IV$ by the naturality of structural isomorphisms $s$ for $F \in \mathcal{Z}\Hecke{1}$. This finishes the proof.  
\end{proof}

\section{Character sheaves as a categorical center}
\label{sec:character-sheaves-as}
\subsection{Monodromic categories.}
In this section we reprove a result of \cite{bezrukavnikovCharacterDmodulesDrinfeld2012} concerning the Drinfeld center of the abelian Hecke category, and extend it to other sheaf-theoretic situations. We also consider some applications mentioned in the introduction.

Let $A$ be a (split) torus. Let $A^\vee_\mathbb{k}$ be the dual $\mathbb{k}$-torus, i.e. algebraic torus over $\mathbb{k}$ whose character lattice
is the cocharacter lattice of $A$. In settings \ref{set:c}-\ref{set:p} let $\mathcal{C}(A)$ be the set of finite order $\mathbb{k}$-points of $A^\vee_\mathbb{k}$. In setting \ref{set:adic} let $\mathcal{C}(A)$ be the group of isomorphism classes of tame rank one local systems on $A$. Finally, in setting \ref{set:mixed} $\mathcal{C}(A) = \mathcal{C}(A \times_{\Spec\Fq} {\Spec\Fqbar})^{\Fr}$, where $\Fr$ stands for the natural Frobenius action on $\mathcal{C}(A \times_{\Spec\Fq} {\Spec\Fqbar})$. In any setting $\lambda \in \mathcal{C}(A)$ defines a rank one (tame) local system with the coefficients in $\mathbb{k}$, which we call $\mathcal{L}_\lambda$.



For a scheme $X$ with an action of $A$, $\lambda \in \mathcal{C}(A),$ let $D^b(X\mondash{\lambda} A)$ stand for the $\lambda$-monodromic subcategory of $D^b(X)$. Here by the $\lambda$-monodromic subcategory we mean the full triangulated idempotent-complete subcategory of $D^b(X)$ generated by $\mathcal{L}_\lambda$-equivariant objects. 

Note that monodrmoic categories for different $\lambda$ are orthogonal inside $D^b(X)$. For a finite subset $S \subset \mathcal{C}(A)$, let $D^b(X\mondash{S} A)$ stand for the full subcategory of $D^b(X)$ consisting of objects that are direct sums of $\lambda$-monodromic objects for some $\lambda \in S$.  We have an equivalence of categories
\[
  D^b(X\mondash{S} A) \simeq \bigoplus_{\lambda \in S} D^b(X\mondash{\lambda} A).
\]

When $X$ is equipped also with a commuting action of an algebraic group $G$, $\lambda \in \mathcal{C}(A)$ (respectively, $S \subset \mathcal{C}(A)$) we write $D^b(G\backslash X\mondash{\lambda} A)$ (respectively, $D^b(G\backslash X\mondash{S} A)$) for the full triangulated subcategory of $D^b(G\backslash X)$ consisting of objects that are in $D^b(X \mondash{\lambda} A)$ (respectively, $D^b(X \mondash{S} A)$) after forgetting the $G$-equivariance.

When $\lambda = 1, \mathcal{L}_{\lambda} \simeq \cc$, we write simply $D^b(X\rightdash A)$ instead of $D^b(X\mondash{1} A)$ and call the corresponding categories ``unipotently monodromic''.

\subsection{Monodromic Hecke categories.}
Recall that in \cite{bezrukavnikovKoszulDualityKacMoody2013} a completed unipotently monodromic Hecke category $\hat{\M} = \hat{D}^b(G\backslash ((G/U)^2\rightdash T^2))$ is defined in cases \ref{set:adic}, \ref{set:mixed}, and in other cases in \cite{bezrukavnikovTopologicalApproachSoergel2018}. The case \ref{set:dmod} is obtained by Riemann-Hilbert correspondence from \ref{set:c}, since $D$-modules that appear are regular holonomic, see \cite{ginzburgAdmissibleModulesSymmetric1989}. 

This definition was extended in the thesis of Gouttard \cite{gouttardPerverseMonodromicSheaves2021} to the case of arbitrary monodromy. We denote by $\hM{\mu}{\lambda}$ the completed monodromic Hecke category with left monodromy $\mu$ and right monodromy $\lambda$. 
Objects of the completed category are certain pro-objects in the category \[{}_{\mu}\M_{\lambda} = {D}^b(G\backslash ((G/U)^2\mondash{(\mu,\lambda)} T^2)).\]

The categories $\hM{\mu}{\lambda}$ comes with two collections of objects indexed by the elements $w \in W$ of the Weyl group, denoted by \[\DD_{w,\lambda},\NN_{w,\lambda} \in \hM{w\lambda}{\lambda},\]  called, respectively, standard and costandard free-monodromic perverse sheaves.  There is a perverse t-structure on the categories ${}_{\mu}\M_{\lambda}$ and $\hM{\mu}{\lambda}$, whose heart we denote by ${}_{\mu}\Perv_{\lambda}$ and ${}_{\mu}\cPerv_{\lambda}$, respectively.

The direct sum of categories $\hM{\mu}{\lambda}$ is equipped with the monoidal structure via the diagram \eqref{eq:7} without the quotient by the right diagonal torus action, with non-zero products of the form \[\hM{\mu}{\lambda}\times \hM{\lambda}{\nu} \to \hM{\mu}{\nu}.\] This monoidal structure extends the monoidal structure on the non-completed categories. We also denote this monoidal structure by $ \star_1' $. It is convenient to shift this monoidal structure in the monodromic case, so that the monodromic monoidal pro-unit becomes perverse, cf. Proposition \ref{braid_relations} below. We write $- \star_1^{mon} - := - \star_1'-[\dim T]$. We will keep the simplified notation $- \star_1 - $ for this shifted monoidal structure -- this should not cause any confusion. 

Standard and costandard pro-objects satisfy the following properties with respect to convolution:
\begin{proposition}
 \label{braid_relations}
  \begin{enumerate}[label=\alph*),ref=(\alph*)]
  \item\label{item:18} $\dd_{\lambda}:=\DD_{e,\lambda} \simeq \NN_{e,\lambda}$ is the unit of the monoidal structure $\star_1$ on $\hM{\lambda}{\lambda}$.
  \item\label{item:19} $\DD_{v,w\lambda} \star_1 \DD_{w,\lambda} \simeq \DD_{vw,\lambda}, \NN_{v,w\lambda} \star_1 \NN_{w,\lambda} \simeq \NN_{vw,\lambda}$ if $l(vw) = l(v) + l(w)$.
  \item\label{item:20} $\DD_{v,v^{-1}\lambda} \star_1 \NN_{v^{-1},\lambda} \simeq \dd_{\lambda}.$
  \end{enumerate}
\end{proposition}
\begin{proof} In unipotent case,  \ref{item:18} and \ref{item:19} are \cite[Lemma 4.3.3]{bezrukavnikovKoszulDualityKacMoody2013}, and \cite[Lemma 6.7]{bezrukavnikovTopologicalApproachSoergel2018} in the setting \ref{set:p}, and \ref{item:20} is  \cite[Lemma 7.7]{bezrukavnikovTopologicalApproachSoergel2018}. For general monodromy, see \cite[Lemmas 8.4.1, 8.4.3]{gouttardPerverseMonodromicSheaves2021}.
\end{proof}
\subsection{Equivariant monodromic categories and character sheaves.}
Let
\[\Hecke{1}_{\lambda} = D^b(G \backslash \yy \mondash{(\lambda,\lambda)}T^{2})\]
be the full $\lambda$-monodromic subcategory of $\Hecke{1}$, and let
\[\Hecke{2}_{\mu,\lambda}=D^b(G^2 \backslash \yy^{(2)}\mondash{(\lambda,\mu,\mu,\lambda)}T^4)\]
be the full $(\lambda, \mu)$-monodromic subcategory of $\Hecke{2}$, with monodromy $\lambda$ along the factors 1, 4 and monodromy $\mu$ along the factors 2, 3.  Note that $\Hecke{1}_{\lambda}$ is an equivariant version of the category ${}_{\lambda}\M_{\lambda}$. For a $W$-orbit $\mathfrak{o}$ in $\mathcal{C}(T)$, we write $\Hecke{1}_{\mathfrak{o}},\Hecke{2}_{\mathfrak{o}}$ for the full subcategories of monodromic sheaves in $\Hecke{1},\Hecke{2}$ with monodromies in $\mathfrak{o}$. We have equivalences
\[
\Hecke{1}_{\mathfrak{o}} \simeq \bigoplus_{\lambda \in \mathfrak{o}}\Hecke{1}_{\lambda},
\]
\[
\Hecke{2}_{\mathfrak{o}} \simeq \bigoplus_{\lambda,\mu \in \mathfrak{o}}\Hecke{2}_{\lambda,\mu}.
\]
Let $D^b_{\mathfrak{C}_{\mathfrak{o}}}(G) \subset D^b(G/_{\Ad}G)$ be the full triangulated subcategory with objects $\F$ satisfying $\hc(\F) \in \Hecke{1}_{\mathfrak{o}}$. Since the functor $\hc$ is monoidal and 
$\Hecke{1}_{\mathfrak{o}}$ is closed under convolution, we conclude that 
$D_{\mathfrak{C}_{\mathfrak{o}}}^b(G)$ is closed under convolution.  The triangulated category $D^b_{\mathfrak{C}_{\mathfrak{o}}}(G)$ is known as the derived category of character sheaves with semisimple parameter $\mathfrak{o}$, and its abelian heart with respect to the perverse t-structure is the category of character sheaves with semisimple parameter $\mathfrak{o}$ (see \cite{lusztigCharacterSheaves1985}, \cite{mirkovicCharacteristicVarietiesCharacter1988}). Let $\mathfrak{C}_{\mathfrak{o}}$ stand for the abelian subcategory of perverse objects in $D^b_{\mathfrak{C}_{o}}(G)$.

We shift the monoidal structure $\star_2$ and action $\bowtie$ by a homological shift $[2\dim T]$, and the monoidal structure $\star$ by $[\dim T]$ (cf. the discussion before Proposition \ref{braid_relations}), keeping the notations the same.

The pro-unit $\dd_\lambda$ of $\hM{\lambda}{\lambda}$ is perverse and $T$-equivariant, so defines the pro-unit in $\Hecke{1}_{\lambda}$, which we denote in the same way.  For a $W$-orbit $\mathfrak{o}$ in $\mathcal{C}(T)$ we write $\dd_{\mathfrak{o}} = \oplus_{\lambda \in \mathfrak{o}}\dd_{\lambda}$. $\dd_{\mathfrak{o}}$ is the monoidal pro-unit in $\Hecke{1}_{\mathfrak{o}}$. 
\begin{definition}
  \label{sec:monodr-categ-char-7}
  For $\lambda \in \mathcal{C}(T)$, write \[\mathcal{Z}\Hecke{1}_{\lambda}:=\operatorname{Fun}_{\Hecke{2}_{\lambda,\lambda}}^{fd}(\Hecke{1}_{\lambda},\Hecke{1}_{\lambda}),\] where \[\operatorname{Fun}_{\Hecke{2}_{\lambda,\lambda}}^{fd}(\Hecke{1}_{\lambda},\Hecke{1}_{\lambda})\] stands for the full subcategory of  $\operatorname{Fun}_{\Hecke{2}_{\lambda,\lambda}}(\Hecke{1}_{\lambda},\Hecke{1}_{\lambda})$ consisting of functors $F$ such that the limit $F(\dd_\lambda)$ exists in $\Hecke{1}_{\lambda}$. 

  For a $W$-orbit $\mathfrak{o}$ in $\mathcal{C}(T)$, write 
  \[\mathcal{Z}\Hecke{1}_{\mathfrak{o}}:=\operatorname{Fun}_{\Hecke{2}_{\mathfrak{o}}}^{fd}(\Hecke{1}_{\mathfrak{o}},\Hecke{1}_{\mathfrak{o}}),\] where \[\operatorname{Fun}_{\Hecke{2}_{\mathfrak{o}}}^{fd}(\Hecke{1}_{\mathfrak{o}},\Hecke{1}_{\mathfrak{o}})\] stands for the full subcategory of  $\operatorname{Fun}_{\Hecke{2}_{\mathfrak{o}}}(\Hecke{1}_{\mathfrak{o}},\Hecke{1}_{\mathfrak{o}})$ consisting of functors $F$ such that the limit $F(\dd_\mathfrak{o})$ exists in $\Hecke{1}_{\mathfrak{o}}$. 
\end{definition}
We have the following variant of Theorem \ref{sec:hecke-categ-char}. Recall that semigroupal category is the category with product satisfying all the axioms of the monoidal category, except those involving the monoidal unit.
\begin{theorem}
  \label{sec:monodr-categ-char}
 Let $\mathfrak{o} = W\lambda$ for $\lambda \in \mathcal{C}(T)$. In all cases \ref{set:c}-\ref{set:mixed}, the functor $\tilde{a}$ induces an equivalence of semigroupal categories 
 \[\tilde{a}_{\mathfrak{C}_{\mathfrak{o}}}:D^b_{\mathfrak{C}_{\mathfrak{o}}}(G) \to \mathcal{Z}\Hecke{1}_{\lambda}.\]
\end{theorem}
\begin{bremark}  
  Note that it is not obvious that the category on the right is closed under the composition. This will become apparent after the proof.
\end{bremark}
The Theorem \ref{sec:monodr-categ-char} evidently follows from  Theorem \ref{sec:equiv-monodr-categ} and Proposition \ref{sec:equiv-monodr-categ-1} below:
\begin{theorem}
  \label{sec:equiv-monodr-categ}
 Let $\mathfrak{o} = W\lambda$ for $\lambda \in \mathcal{C}(T)$. In all cases \ref{set:c}-\ref{set:mixed}, the functor $\tilde{a}$ induces an equivalence of semigroupal categories 
 \[\tilde{a}_{\mathfrak{C}_{\mathfrak{o}}}:D^b_{\mathfrak{C}_{\mathfrak{o}}}(G) \to \mathcal{Z}\Hecke{1}_{\mathfrak{o}}.\]
\end{theorem}
\begin{proposition}
  \label{sec:equiv-monodr-categ-1}
 Let $\mathfrak{o} = W\lambda$ for $\lambda \in \mathcal{C}(T)$. The projection \[r_\lambda:\mathcal{Z}\Hecke{1}_{\mathfrak{o}} \to \mathcal{Z}\Hecke{1}_{\mathfrak{\lambda}}\] is a monoidal equivalence. 
\end{proposition}
\begin{proof}[Proof of Proposition \ref{sec:equiv-monodr-categ-1}]
This is similar to \cite[Lemma 11.12]{lusztigEndoscopyHeckeCategories2020a}. For $\mu, \lambda\in \mathfrak{o}$, let \[W_{\mu,\lambda} = \{w \in W, w\mu = \lambda\}.\] Recall a notion of a minimal element in $W_{\mu,\lambda}$ from \cite[Lemma 4.2]{lusztigEndoscopyHeckeCategories2020a}. It is known that if $w \in W_{\mu,\lambda}$ is a minimal element, one has that the canonical map
  \begin{equation}
    \label{eq:10}
    \DD_{w, \lambda} \to \NN_{w, \lambda}
  \end{equation}
  is an isomorphism, see \cite[Corollary 8.5.5]{gouttardPerverseMonodromicSheaves2021}, and that if $w' \in W_{\nu, \mu}$ is a minimal element, then $ww' \in W_{\nu,\lambda}$ is again a minimal element, see \cite[Corollary 4.3]{lusztigEndoscopyHeckeCategories2020a}.

  Let $\pi_{23}$ stand for the projection
  \[
   \pi_{23}:G^2\backslash (G/U)^4 \to G^2\backslash (G/U)^4/T, 
 \]
 where the action of $T$ on the right is along the factors 2, 3.
  For $w$ a minimal element in $W_{\mu,\lambda},$ let
  \[\bar{\Xi}_{\lambda,\mu}^w = \pi_{23!}(\DD_{w,\lambda}\boxtimes\DD_{w^{-1},\mu}).\]
  This object is perverse and $T$-equivariant with respect to the action of $T$ along the factors 1, 4, and so descends to an object
\[\Xi_{\lambda,\mu}^w \in \Hecke{2}_{\lambda,\mu}.\]

  It follows from diagram chase, Proposition \ref{braid_relations} and the fact that \eqref{eq:10} is an isomorphism that the functor $\Xi_{\lambda,\mu}\bowtie -$ is an equivalence of categories $\Hecke{1}_{\lambda} \to \Hecke{1}_\mu.$

  For $F \in \mathcal{Z}\Hecke{1}_{\lambda}, \mu \in \mathfrak{o}$, define a functor $F_\mu:\Hecke{1}_{\mu} \to \Hecke{1}_{\mu}$ by
  \[F_\mu(A) := \Xi_{\lambda,\mu}^{w}\bowtie F(\Xi_{\mu,\lambda}^{w^{-1}}\bowtie A).\]
  Using the central structure on $F$, we see that $F_\mu$ does not depend on the choice of minimal elements and isomorphism classes of standard pro-objects up to a canonical isomorphism. It follows that the assignment
  \[
    F \mapsto \oplus_{\mu}F_\mu
  \]
  defines a functor inverse to $r_{\lambda}$.
\end{proof}
The rest of the section is occupied with the proof of Theorem \ref{sec:equiv-monodr-categ}. We will need the following:
\begin{proposition}
 \label{sec:monodr-categ-char-5} Let $\mathfrak{Z}$ be an object in $\Hecke{1}_{\mathfrak{o}}$ such that the convolution functor $- \star_1 \mathfrak{Z}$ can be given a structure of the module endofunctor in $\mathcal{Z}\Hecke{1}_{\mathfrak{o}}$. Then the functor $ - \star_1 \mathfrak{Z}$, considered as a  functor from $\operatorname{Fun}(\Hecke{1},\Hecke{1})$ with values in $\Hecke{1}_{\mathfrak{o}}$, can be given a structure of a module endofunctor in $\mathcal{Z}\Hecke{1}$, compatible with its structure of a functor in $\mathcal{Z}\Hecke{1}_{\mathfrak{o}}$ on monodromic objects in $\Hecke{2}$. 
\end{proposition}
This will be in turn deduced using the following simple statement.
\begin{lemma}
  \label{sec:equiv-monodr-categ-2}
 Let $\cC, \cD$ be categories and let $\cC'$ be a full subcategory of $\cC$. Let $F_1, F_2 : \cC \to \cD $ be two functors, such that their restriction to $\cC'$ are isomorphic. Assume that $G_1, G_2$ are right adjoint to $F_1, F_2$ respectively. If $G_1, G_2$ take values in $\cC'$, then $F_1 \simeq F_2$, with isomorphism compatible with the chosen isomorphism of their restrictions to $\cC'$. 
\end{lemma}
\begin{proof}
 Indeed, since $G_1, G_2$ take values in $\cC'$, they are right adjoint to restrictions of $F_1, F_2$ to $\cC'$, respectively. Since the two restrictions are isomorphic, it follows that $G_1 \simeq G_2$, and so $F_1 \simeq F_2$. 
\end{proof}
\begin{proof}[Proof of Proposition \ref{sec:monodr-categ-char-5}]
  We need to construct, for all objects $\cA \in \Hecke{1}, \cB \in \Hecke{2},$ a natural isomorphism
  \[
    \cB \bowtie (\cA \star_1 \mathfrak{Z}) \rsim (\cB \bowtie \cA)\star_1 \mathfrak{Z}.
  \]
   Define two functors $\Hecke{2} \to \Hecke{1}_{\mathfrak{o}}$ as
  \[
    F_1(\cB) = \cB \bowtie (\cA \star_1 \mathfrak{Z}), F_2(\cB) = (\cB \bowtie \cA) \star_1 \mathfrak{Z}.
  \]
  Since $-\star_1 \mathfrak{Z}$ has a structure of the functor in $\mathcal{Z}\Hecke{1}_{\mathfrak{o}}$, the functor $F_1$ is isomorphic to $F_2$ when restricted to $\Hecke{1}_{\mathfrak{o}}$. Hence, by Lemma \ref{sec:equiv-monodr-categ-2}, it is enough to show that the functors right adjoint to $F_1, F_2$ take values in $\Hecke{2}_{\mathfrak{o}}$.

It is easy to see that the functor right adjoint to $-\star_1\mathfrak{X}, \mathfrak{X} \in \Hecke{1},$ preserves the category $\Hecke{1}_{\mathfrak{o}}$ (this functor is given by dually defined ``$*$-convolution'' with $\D\mathfrak{X}$). It is thus enough to show that the functor right adjoint to $\cB \mapsto \cB \bowtie \cA$ takes values in $\Hecke{2}_{\mathfrak{o}}$ for $\cA \in \Hecke{1}_{\mathfrak{o}}$.

  Recall that the definition of the functor $- \bowtie -$. We have a diagram
\[
  \begin{tikzcd}
    & \yy^{(2)} \arrow[ld,"p_{23}"']\arrow[rd,"p_{14}"] & \\
\yy &                                      & \yy 
  \end{tikzcd}
\]
For $\mathcal{B} \in \Hecke{2},\mathcal{A}\in\Hecke{1},$ we defined
\[
  \mathcal{B}\bowtie\mathcal{A} = p_{14!}(\operatorname{For}^{G^2}_G(\mathcal{B}) \otimes p_{23}^*(\mathcal{A})).
\]
It follows that the right adjoint is given by the following expression, up to the omitted shift:
\[
 \F \mapsto \operatorname{Av}_{G \times G*}^Gp_{14}^*\F \overset{!}\otimes p_{23}^*\D\cA,  
\]
where $\operatorname{Av}_{G \times G*}^G$ stands for the operator of $*$-averaging from $G$-equivariant derived category to $G \times G$-equivariant derived category. Since $\cA$ is assumed to be monodromic, the sheaf  $p_{14}^*\F \overset{!}\otimes p_{23}^*\D\cA$ is monodromic with respect to the projection $G/U \to G/B$ along the second and third factors. Since the left action of $G \times G$ commutes with the right $T$-action along every factor, it follows that $\operatorname{Av}_{G \times G*}^Gp_{14}^*\F \overset{!}\otimes p_{23}^*\D\cA$ is also monodromic with respect to the corresponding projection. The claim now follows, since the complex in $D^b(G\backslash (G/U \times G/U))$ monodromic with respect to the projection along one factor is also monodromic with respect to the projection along the other.
\end{proof}
\begin{proof}[Proof of Theorem \ref{sec:monodr-categ-char}.]
  Let $(\mathcal{Z}\Hecke{1})_{\mathfrak{o}}$ be the full subcategory of $\mathcal{Z}\Hecke{1}$ with objects $F$ satisfying $\varepsilon(F) \in \Hecke{1}_{\mathfrak{o}}$. Since $\varepsilon$ is monoidal, the subcategory $(\mathcal{Z}\Hecke{1})_{\mathfrak{o}}$ is semigroupal. 

Now it follows from Theorem \ref{sec:hecke-categ-char} and 
Proposition \ref{sec:sheaves-group-as-1} that $\tilde{a}$ induces an equivalence of semigroupal categories 
\[
D_{\mathfrak{C}}^b(G) \to (\mathcal{Z}\Hecke{1})_{\mathfrak{o}}. 
\]

It suffices to show that there exists a natural equivalence of semigroupal categories 
\[(\mathcal{Z}\Hecke{1})_{\mathfrak{o}}\rsim \mathcal{Z}\Hecke{1}_{\mathfrak{o}}.\] 

The functor $(\mathcal{Z}\Hecke{1})_{\mathfrak{o}}\to \mathcal{Z}\Hecke{1}_{\mathfrak{o}}$ is given by forgetting the module structure for the action of non-monodromic objects. For $F \in (\mathcal{Z}\Hecke{1})_{\mathfrak{o}}$ the condition on the limit $F(\dd_{\mathfrak{o}})$ to exist follows from $\dd_{\mathfrak{o}}$ being a monoidal pro-unit, see Proposition \ref{braid_relations} \ref{item:18}.

To define the inverse functor, we need to show that any functor in $F_{\mathfrak{o}} \in \mathcal{Z}\Hecke{1}_{\mathfrak{o}}$ can be uniquely extended to a functor $F \in \mathcal{Z}\Hecke{1}$, with \[F_{\mathfrak{o}} \simeq F\circ \iota_{\mathfrak{o}},\] where $\iota_{\mathfrak{o}}$ is the canonical embedding $\Hecke{1}_{\mathfrak{o}}\to\Hecke{1}$.

First note that, since $F_{\mathfrak{o}}(\dd_{\mathfrak{o}})$ exists in $\Hecke{1}_{\mathfrak{o}}$, we have \[F_{\mathfrak{o}}(\cA) \simeq \cA \star_1 F_{\mathfrak{o}}(\dd_{\mathfrak{o}}).\] Indeed, since in the monodromic category $*$-convolution is isomorphic, up to a homological shift, to the $!$-convolution (see, for example, \cite[Lemma 4.1.1]{hhh}), the functor $\cA$ is right adjoint and so preserves limits. We thus have, using Lemma \ref{sec:centr-categ-as} and the fact that $\dd_{\mathfrak{o}}$ is a monoidal unit
\begin{multline}
  \cA \star_1 F_{\mathfrak{o}}(\dd_{\mathfrak{o}}) \simeq F_{\mathfrak{o}}(\cA \star_1 \dd_{\mathfrak{o}}) \simeq F_{\mathfrak{o}}(\cA) \star_1 \dd_{\mathfrak{o}} \simeq F_{\mathfrak{o}}(\cA). 
\end{multline}
The result now follows from Proposition \ref{sec:monodr-categ-char-5} applied to $\mathfrak{Z} = F_{\mathfrak{o}}(\dd_{\mathfrak{o}}).$
\end{proof}
\subsection{Free-monodromic tilting sheaves.}\label{sec:free-monodr-tilt}

\emph{Until the end of this section, we work in the settings \ref{set:c}-\ref{set:adic}.}

Recall that an object in $\hM{\mu}{\lambda}$ is called a free-monodromic tilting object, if it can be both obtained by successive extensions of standard free-monodromic objects and by successive extensions of costandard free-monodromic objects. 

We have the following standard corollary of Proposition \ref{braid_relations}.
\begin{corollary}
  \label{sec:some-facts-about}
  Let $T$ be a free-monodromic tilting object in $\hM{\mu}{\lambda}$.
  \begin{enumerate}[label=\alph*),ref=(\alph*)]
  \item\label{item:7} The convolution functors $T\star_1 -,-\star_1T$ are t-exact with respect to the perverse t-structure.
  \item\label{item:8} The objects $\DD_{w_0,\mu}\star_1 T \star_1 \NN_{w_0,w_0\lambda}, \NN_{w_0,\mu}\star_1 T\star_1 \DD_{w_0,w_0\lambda}$ are free-monodromic tilting. 
  \end{enumerate} 
\end{corollary}

\subsection{DG models for monodromic categories.}\label{sec:dg-models-monodromic-1}
Define the shifted  perverse t-structure on ${}_{\mu}\M_\lambda$ and its heart by
\[
  {}_{\mu}D^{\leq 0}_{w_0,\lambda} = \DD_{w_0,w_0\mu}\star_1 {}^pD^{\leq 0},{}_{\mu}D^{\geq 0}_{w_0,\lambda} = \DD_{w_0,w_0\mu}\star_1 {}^pD^{\geq 0},
\]
\[
    {}_{\mu}\Perv_{w_0,\lambda} := \DD_{w_0,w_0\mu}\star_1{}_{w_0\mu}\Perv_{\lambda}.
\]
Here ${}^pD^{\leq 0}, {}^pD^{\geq 0}$ stand for the perverse t-structure on ${}_{w_0\mu}\M_{\lambda}$.

Let $X$ be a scheme with a left action of an algebraic group $G$ and commuting right action of an algebraic torus $A$. Let $S_{A,\lambda}$ be the completed local ring at $\lambda$ of $A^\vee_\mathbb{k}$ in settings  \ref{set:c}-\ref{set:p} or the  moduli space of tame local systems on $A$ in the setting \ref{set:adic}. Let $K_{A,\lambda}$ be the Koszul complex of the maximal ideal in $S_{A,\lambda}$. We may identify $S_{A,\lambda}$ with the symmetric algebra of the $\mathbb{k}$-vector space $V_{A,\lambda}$ completed at the origin. 

We will need some explicit DG models for the derived $\Hecke{1}_{\lambda}$ and $\Hecke{2}_{\lambda,\lambda}$, which we now describe.

Let $\T^{(1)}_{\lambda,DG}$ be the category of complexes of free-monodromic tilting sheaves in ${}_\lambda\hat{\M}_\lambda$ together with the action of the DG-algebra $K_{T,\lambda}$ for the right diagonal action of $T$ on $Y^2$, with action of $S_{T,\lambda}$ compatible with the monodromy action, and such that the underlying complex represents an object in ${}_\lambda\M_\lambda$.

Let $\mathcal{P}^{(1)}_{\lambda,DG}$ be the category of complexes of perverse sheaves in ${}_\lambda\M_\lambda$ together with the action of the DG-algebra $K_{T,\lambda}$ for the right diagonal action of $T$ on $Y^2$, with action of $S_{T,\lambda}$ compatible with the monodromy action.

Let $\mathcal{P}_{w_0,\lambda,DG}^{(1)}$ be the category of complexes of objects in ${}_\lambda\mathcal{P}_{w_0,\lambda}$ together with the action of the DG-algebra $K_{T,\lambda}$ for the right diagonal action of $T$ on $Y^2$, with action of $S_{T,\lambda}$ compatible with the monodromy action.


Let ${\T}^{(2)}_{\lambda,\lambda,DG}$ be the category of complexes of free-monodromic tilting sheaves in ${}_{\lambda,\lambda}\hat{\M}^{(2)}_{\lambda,\lambda}:=\hat{D}^b(G^2 \backslash ((G/U)^4\mondash{(\lambda,\lambda,\lambda,\lambda)}T^4)$ together with the action of the DG-algebra $K_{T^2,(\lambda,\lambda)}$ for the action of $T^2$ on $Y^4$ given by the formula \eqref{eq:6}, with the action of $S_{T^2,(\lambda,\lambda)}$ compatible with the monodromy action, and such that the underlying complex represents an object in ${}_{\lambda,\lambda}\M^{(2)}_{\lambda,\lambda}:={D}^b(G^2 \backslash ((G/U)^4\mondash{(\lambda,\lambda,\lambda,\lambda)}T^4)$. 

We define the convolutions $-\star^{DG}_1-$ between $\T^{(1)}_{\lambda,DG}$ and either $\T^{(1)}_{\lambda,DG}, \mathcal{P}^{(1)}_{\lambda,DG}$ or  $\mathcal{P}_{w_0,\lambda,DG}^{(1)}$; $-\star^{DG}_2-$ between ${\T}^{(2)}_{\lambda,\lambda,DG}$ and
${\T}^{(2)}_{\lambda,\lambda,DG}$;
and   ${-\bowtie^{DG}-}$ between ${\T}^{(2)}_{\lambda,\lambda,DG}$ and $\T^{(1)}_{\lambda,DG}$ 
as follows. The convolution between two  free-monodromic tilting objects $\mathcal{T}_1\star_1\mathcal{T}_2= p_{13!}(p_{12}^*\mathcal{T}_1\otimes p_{23}^*\mathcal{T}_2)$ is defined via diagram
\[
  \begin{tikzcd}
    & Y^3 \arrow[ld,"p_{12}"']\arrow[rd,"p_{23}"]\arrow[r,"p_{13}"]  & Y^2 \\
Y^2  &                                      & Y^2, 
  \end{tikzcd}
\]
and is again a tilting object as follows from Proposition \ref{braid_relations}. Considering the sheaves on $Y^3$ equipped with the action of $K_{T,\lambda}$ corresponding to the diagonal action of $T$ makes the operations compatible with the $K_{T,\lambda}$-action, thus providing an action of $K_{T,\lambda}$ on $\mathcal{T}_1\star_1\mathcal{T}_2$.  This gives a convolution of two objects in $\T^{(1)}_{\lambda,DG}$. We define $P_1\star_1^{DG}\mathcal{T}_2$ with $P_1$ in ${}_\lambda\mathcal{P}_{\lambda}$ or ${}_\lambda\mathcal{P}_{w_0,\lambda}$ and $\mathcal{T}_2$ free-monodromic tilting as an object in $\mathcal{P}^{(1)}_{\lambda,DG}$ or $\mathcal{P}_{w_0,\lambda,DG}^{(1)}$ in the same way by applying Corollary \ref{sec:some-facts-about}. 

To define $\mathcal{T}_1\star_2^{DG} \mathcal{T}_2$ in ${\T}^{(2)}_{\lambda,\lambda,DG}$  we start with taking \[
\mathcal{T}_1\star'_2\mathcal{T}_2 = p_{1346!}(p_{1256}^*\mathcal{T}_1\otimes p_{2345}^*\mathcal{T}_2).
\] defined via diagram 
\[
  \begin{tikzcd}
    & Y^6 \arrow[ld,"p_{1256}"']\arrow[rd,"p_{2345}"]\arrow[r,"p_{1346}"]  & Y^4 \\
Y^4 &                                      & Y^4.
  \end{tikzcd}
\] 
This complex is equipped with the action of $K_{T^2,(\lambda,\lambda)}$ compatible with the action of the monodromy where the one copy of $K_{T,\lambda}$ acts via its action on the first and the fourth coordinates on $\mathcal{T}_1$ and the other via its action on the second and the third coordinates on $\mathcal{T}_2$. Furthermore, note that the action of $K_{T,\lambda}$ via the action on the second and the third coordinates on $\mathcal{T}_1$ and the action of $K_{T,\lambda}$ via the action on the first and the fourth coordinates on $\mathcal{T}_2$ induce the same action on $\mathcal{T}_1\star'_2\mathcal{T}_2$. The diagonal $K_{T,\lambda}$-action kills the augmentation ideal of its degree zero component $S_{T,\lambda}=K_{T,\lambda}^0$ and, hence, factors through the action of the exterior algebra $\Lambda(V_{T,\lambda}[1])$ of the Koszul generators.
We define 
 \[\mathcal{T}_1\star^{DG}_2\mathcal{T}_2:=(\mathcal{T}_1\star'_2 \mathcal{T}_2)\otimes^{L}_{\Lambda(V_{T,\lambda}[1])}\mathbb{k},\]
 where $\mathbb{k}$ is considered to be the 1-dimensional module over $\Lambda(V_{T,\lambda}[1])$ concentrated in homological degree 0. Together with the action of $K_{T^2,(\lambda,\lambda)}$ from the above this provides an object of ${\T}^{(2)}_{\lambda,\lambda,DG}$.
 
 Similarly,  to define $\mathcal{T}_1\bowtie^{DG} \mathcal{T}_2$ we start with $
  \mathcal{T}_1\bowtie' \mathcal{T}_2 = p_{14!}(\mathcal{T}_1 \otimes p_{23}^*(\mathcal{T}_2))$ given via
   \[
  \begin{tikzcd}
    & Y^4 \arrow[ld,"p_{23}"']\arrow[rd,"p_{14}"] & \\
Y^2 &                                      & Y^2.
  \end{tikzcd}
\]
The action of $K_{T,\lambda}$ on $\mathcal{T}_1\bowtie' \mathcal{T}_2$ comes from the action on $\mathcal{T}_1$ via the action on the first and the fourth coordinates. Meanwhile, the action of $K_{T,\lambda}$ on  $\mathcal{T}_1$ via the action on the second and the third coordinates and the action of $K_{T,\lambda}$ on $\mathcal{T}_2$ induce the same action on $\mathcal{T}_1\bowtie'\mathcal{T}_2$. The diagonal $K_{T,\lambda}$-action kills the augmentation ideal of its degree zero component $S_{T,\lambda}=K_{T,\lambda}^0$ and, hence, factors through the exterior algebra $\Lambda(V_{T,\lambda}[1])$. We define 
 \[\mathcal{T}_1\bowtie^{DG}\mathcal{T}_2:=(\mathcal{T}_1\bowtie'_2 \mathcal{T}_2)\otimes^{L}_{\Lambda(V_{T.\lambda}[1])}\mathbb{k}.\]
 Together with the action of $K_{T,\lambda}$ from the above this provides an object of $\T^{(1)}_{\lambda,DG}$.

We have the following
\begin{proposition}
 \label{sec:dg-models-mixed} 
 \begin{enumerate}[label=\alph*),ref=(\alph*)]
 \item The DG categories $\T^{(1)}_{\lambda,DG}, \mathcal{P}^{(1)}_{\lambda,DG}, \mathcal{P}_{w_0,\lambda,DG}^{(1)}$ are DG models for the category $\Hecke{1}_{\lambda}$.
 \item The DG categories $
  {\T}^{(2)}_{\lambda,\lambda,DG}$ is a DG model for the category $\Hecke{2}_{\lambda,\lambda}$.
 \item The convolutions $-\star^{DG}_1-,-\star^{DG}_2-$ and $-\bowtie^{DG}-$ define monoidal structures on the respective categories or actions of the monoidal categories. Moreover, the DG enhancements $\T^{(1)}_{\lambda,DG}, \mathcal{P}^{(1)}_{\lambda,DG},  \mathcal{P}_{w_0,\lambda,DG}^{(1)}$ of $\Hecke{1}_{\lambda}$  and ${\T}^{(2)}_{\lambda,\lambda,DG}$
    of $\Hecke{2}_{\lambda,\lambda}$ are compatible with the convolution structures.
 \end{enumerate}
\end{proposition}
\begin{proof}
Parts a) and b) are a straightforward generalization of \cite[Lemma 44, Corollary 45]{bezrukavnikovTwoGeometricRealizations2016}.  Part c) follows from \cite[Proposition 51]{bezrukavnikovTwoGeometricRealizations2016}, similarly to  \cite[Proposition 50]{bezrukavnikovTwoGeometricRealizations2016}.
\end{proof}

\subsection{Abelian category of character sheaves as a categorical center.}\label{sec:abel-categ-char-1}
Since the $\star$-convolution on $D^b(G/_{Ad}G)$ is left t-exact with respect to the perverse t-structure, the abelian category $\mathfrak{C}$ becomes a semigroupal category with product
\[
  X_1 \star^0 X_2 := {}^p\mathcal{H}^0(X_1\star X_2).
\]

We will use the following result, proved in characteristic 0 setting in \cite{bezrukavnikovCharacterDmodulesDrinfeld2012} and in general geometric setting in \cite{chenFormulaGeometricJacquet2017}. Note that the latter article works only with unipotent monodromy. Its methods can be extended to treat the general case; we also plan to give a new uniform proof in a subsequent publication. Let \[\M_{\mathfrak{o}} = \bigoplus_{\mu_1,\mu_2 \in \mathfrak{o}}{}_{\mu_1}\M_{\mu_2}.\]Let \[\operatorname{For}_T:\Hecke{1}_{\mathfrak{o}}\to\M_{\mathfrak{o}}\] stand for the functor forgetting the $T$-equivariance.

\begin{proposition}[\cite{bezrukavnikovCharacterDmodulesDrinfeld2012} Corollary 3.4, \cite{chenFormulaGeometricJacquet2017} Corollary 7.9]
\label{sec:monodr-categ-char-2} The functor \[\F \mapsto \NN_{w_0,\lambda}\star\operatorname{For}_T\hc(\F), D^b_{\mathfrak{C}_{\mathfrak{o}}}(G) \to \M_{\mathfrak{o}}\] is t-exact with respect to the perverse t-structure.
\end{proposition}

\begin{proposition}
  \label{sec:some-facts-about-1}
  The convolution $\star_1:{}_{\lambda}\M_\lambda \times {}_{\lambda}\M_\lambda \to {}_{\lambda}\M_\lambda$ is left t-exact with respect to the shifted t-structure $({}_{\lambda}D^{\leq 0}_{w_0,\lambda}, {}_{\lambda}D^{\geq 0}_{w_0,\lambda})$.
\end{proposition}

\begin{proof}
  This follows by a standard argument from the proof of \cite[Proposition 4.3.4]{bezrukavnikovKoszulDualityKacMoody2013}, using the exactness properties of convolution with the standard and costandard sheaves, cf. also \cite[Appendix 5.1]{arkhipovPerverseSheavesAffine2020}.
\end{proof}

For completeness, in the Appendix A we give another proof of Proposition \ref{sec:some-facts-about-1}, following an idea of V. Drinfeld, which is more explicit and does not use the exactness properties of convolution with the standard and costandard sheaves.

Denote by $\mathcal{H}_{w_0,\lambda}^\bullet(-)$ the cohomology functor with respect to the t-structure $({}_{\lambda}D^{\leq 0}_{w_0,\lambda}, {}_{\lambda}D^{\geq 0}_{w_0,\lambda})$. By Proposition \ref{sec:some-facts-about-1} we can equip the abelian category ${}_{\lambda}\Perv_{w_0,\lambda}$ with the truncated monoidal structure using the formula
\[
  X_1 \star_1^0 X_2 := \mathcal{H}_{w_0,\lambda}^0(X_1\star_1 X_2), X_1, X_2 \in {}_{\lambda}\Perv_{w_0,\lambda}.
\]

Let $\mathcal{Z}_D(\mathcal{A})$ stand for the Drinfeld center of a monoidal category $\mathcal{A}$.
We now give a new proof of \cite[Theorem 3.6]{bezrukavnikovCharacterDmodulesDrinfeld2012} and extend it to other sheaf-theoretic situations.

\begin{theorem}
  \label{sec:abel-categ-char}
Let $\mathfrak{o}=W\lambda$ for $\lambda \in \mathcal{C}(T)$. In settings \ref{set:c}-\ref{set:adic}, restriction of the functor $\operatorname{For}_T\hc$ to $\mathfrak{C}_{\mathfrak{o}}$ can be lifted to the equivalence of braided semigroupal categories \[\mathfrak{C}_{\mathfrak{o}} \rsim \mathcal{Z}_D({}_\lambda\Perv_{w_0,\lambda}). \]
\end{theorem}

\begin{proof}[Proof of Theorem \ref{sec:abel-categ-char}]
  Let $\mathcal{Z}\Hecke{1}_{ab,\mathfrak{o}}$ be the full subcategory of functors $F \in \mathcal{Z}\Hecke{1}_{\mathfrak{o}}$ satisfying $\operatorname{For}_T \circ \varepsilon(F) \in \oplus_{\mu \in \mathfrak{o}} {}_\lambda\Perv_{w_0,\mu}$. It follows from Theorem \ref{sec:monodr-categ-char} and Proposition \ref{sec:monodr-categ-char-2} that we have an equivalence of semigroupal categories $\mathfrak{C} \rsim \mathcal{Z}\Hecke{1}_{ab,\mathfrak{o}},$ induced by the functor $\tilde{a}$. 

  It suffices to show that $\operatorname{For}_T\circ\varepsilon$ induces an equivalence
  \[
    \mathcal{Z}\Hecke{1}_{ab,\mathfrak{o}} \rsim \mathcal{Z}_D({}_\lambda\Perv_{w_0,\lambda}).
  \]

  Indeed, the fact that $\operatorname{For}_T\circ\varepsilon\circ\tilde{a}(\F), \F \in \mathfrak{C}_{\mathfrak{o}},$ has a structure of an object in $\mathcal{Z}_D({}_\lambda\Perv_{w_0,\lambda})$ follows from an argument completely analogous to the one in to Proposition \ref{sec:sheaves-group-as-1} \ref{item:3}.

  We now construct the inverse functor $g:\mathcal{Z}_D({}_\lambda\Perv_{w_0,\lambda}) \to \mathcal{Z}\Hecke{1}_{ab,\mathfrak{o}}$. Take $Z \in \mathcal{Z}_D({}_\lambda\Perv_{w_0,\lambda})$. Looking at the central isomorphism $Z\star_1\dd_\lambda\to\dd_\lambda\star_1 Z$, we see that the left monodromy action on $Z$ coincides with the right one, so that $Z$ can be considered an object $\bar{Z}$ of $\Perv_{w_0,\lambda,DG}^{(1)}$ concentrated in the 0th homological degree. Let $T_1, T_2, T_3$ be three complexes of free monodromic tilting sheaves in ${}_\lambda\hat{\M}_\lambda$. The central structure on $Z$ gives an isomorphism
  \begin{equation}
    \label{eq:14}
    T_1\star_1 T_2 \star_1 Z \star_1 T_3 \to T_1 \star_1 T_2 \star_1 T_3 \star Z.
\end{equation}
  of complexes of objects in ${}_\lambda\Perv_{w_0,\lambda}$, by Corollary \ref{sec:some-facts-about}. Since any indecomposable free-monodromic tilting sheaf in the $\lambda$-monodromic Hecke category for $G^2$ has the form $T_1 \boxtimes T_3$ for free-monodromic tilting sheaves in the $T_1, T_3$  in ${}_\lambda\hat{\M}_\lambda$, isomorphism \eqref{eq:14} defines an isomorphism  
  \[
    \bar{A}\bowtie(\bar{X}\star_1 \bar{Z}) \to (\bar{A}\bowtie \bar{X})\star_1 \bar{Z}
  \]
  in $\Perv_{w_0,\lambda,DG}^{(1)}$ for objects $\bar{A}$ in $\T^{(2)}_{\lambda,\lambda,DG}$, $\bar{X}$ in $\T^{(1)}_{\lambda,DG}$, natural in $\bar{A},\bar{X}$. It follows, by Proposition \ref{sec:dg-models-mixed}, that convolution with $Z$ defines an object in $\mathcal{Z}\Hecke{1}_{ab,\mathfrak{o}}$, which completes the proof.

  The compatibility with the braiding follows by construction of the structure of the module endofunctor on $-\star_1(\hc \circ \tilde{a}^{-1}\circ g(Z))$ in Theorem \ref{sec:hecke-categ-char} (cf. Proposition \ref{sec:sheaves-group-as-1} \ref{item:3}). 
\end{proof}
\subsection{Vanishing result for central sheaves on a torus.}
We return to an arbitrary setting \ref{set:c} - \ref{set:mixed}.

Consider the closed embedding $i\colon T\hookrightarrow \mathcal{Y}$ given by $t\mapsto (U, tU)$, which passes to $i\colon T/_{\mathrm{Ad}}T\hookrightarrow G\backslash \mathcal{Y}$ after taking quotients. 
Note that the category of  the unipotent free-monodromic sheaves on $T$ is equivalent to the category of $S_{T,\lambda}$-modules.
The action of $W$ on $T$ induces an action of the stabilizer $W_{\lambda,\lambda}\subset W$ of $\lambda$ on $S_{T,\lambda}$. . Let $W_\lambda^\circ\subset W_{\lambda,\lambda}$ be the subgroup generated by simple reflections $s_\alpha$ such that $\alpha(\lambda)=1$.

\begin{proposition}
  \label{sec:dg-models-monodromic}
Let $M$ be an $S_{T,\lambda}$-module and let $\mathcal{F}$ be the corresponding sheaf on $T$. Then $i_*\mathcal{F}$ is in ${}_\lambda\mathcal{P}_{w_0,\lambda}$ and admits a central structure in $\mathcal{Z}_D({}_\lambda\mathcal{P}_{w_0,\lambda})$ if and only if $M$ admits an action of the semidirect product ring $W_{\lambda,\lambda}\#S_{T,\lambda}$ and descends to an $S_{T,\lambda}^{W_\lambda^\circ}$-module compatibly with the action of $W_\lambda^\circ\subset W_{\lambda,\lambda}$, i.e. there exists an $S_{T,\lambda}^{W_\lambda^\circ}$-module  $M_0$, such that $M\cong M_0\otimes_{S_{T,\lambda}^{W_\lambda^\circ}}S_{T,\lambda}$ and the action of $W_\lambda^\circ\subset W_{\lambda,\lambda}$ is given by the natural action on the second factor.
\end{proposition}
\begin{proof}
  The fact that $i_*\mathcal{F}$ is in ${}_\lambda\mathcal{P}_{w_0,\lambda}$ follows from base change and t-exactness of the $*$-pushforward from an affine open subset. 

Note that the objects of the form $i_*\mathcal{F}$ lie in the neutral block of ${}_\lambda\hat{\M}_\lambda$ (see \cite[Section 4]{lusztigEndoscopyHeckeCategories2020a} for the definition).  By 
\cite[Theorem 13.1.1.1]{gouttardPerverseMonodromicSheaves2021}
there is a fully faithful monoidal functor from the category of tilting objects in the neutral block of ${}_\lambda\hat{\M}_\lambda$ to the category of ${S_{T,\lambda}\otimes_{S_{T,\lambda}^{W_\lambda^\circ}}S_{T,\lambda}}$-modules with the monoidal structure given by $-\otimes_ {S_{T,\lambda}}-$. The essential image of this functor is the category of Soergel bimodules. By construction, objects supported on the image of $i$ under this functor correspond to the $S_{T,\lambda}\otimes_{S_{T,\lambda}^{W_\lambda^\circ}}S_{T,\lambda}$-modules supported on the diagonal. More precisely, $i_*\mathcal{F}$ goes to $M$ with  $S_{T,\lambda}\otimes_{S_{T,\lambda}^{W_\lambda^\circ}}S_{T,\lambda}$ acting via the multiplication map. 

The other blocks are enumerated by the coset $W_{\lambda,\lambda}/W_\lambda^\circ$ and the subcategory  of tilting objects in the block corresponding to $\beta\in W_{\lambda,\lambda}/W_\lambda^\circ$ are equivalent to the category of Soergel bimodules with the right action of $S_{T,\lambda}$ twisted by the shortest representative of the coset $\beta$ (\cite[Theorem 13.1.2.1]{gouttardPerverseMonodromicSheaves2021}).

Since each objects in ${}_\lambda\mathcal{P}_{w_0,\lambda}$ admits a tilting resolution (as follows, for example, from \cite[Theorem 4.6.1]{gouttardPerverseMonodromicSheaves2021}) it is sufficient to study central Soergel bimodules. If $M$ descends to an $S_{T,\lambda}^{W_\lambda^\circ}$-module $M_0$ and $N$ is any $S_{T,\lambda}\otimes_{S_{T,\lambda}^{W_\lambda^\circ}}S_{T,\lambda}$-module, 
 then there are canonical isomorphisms 
\[
M\otimes_{S_{T,\lambda}}N\simeq M_0\otimes_{S_{T,\lambda}^{W_\lambda^\circ}} N\simeq N\otimes_{S_{T,\lambda}^{W_\lambda^\circ}} M_0\simeq N\otimes_{S_{T,\lambda}}M,
\]
satisfying all required compatibilities. Furthermore, the $W_{\lambda,\lambda}$ action on $M$ provides the central morphism with $N$ being the diagonal bimodule over $S_{T,\lambda}$ with the right action twisted by an element $\beta\in W_{\lambda,\lambda}$. Together these are sufficient to endow $i_*\mathcal{F}$ with the desired central structure.

Conversely, if $i_*\mathcal{F}$ admits a central structure, since $S_{T,\lambda}\otimes_{S_{T,\lambda}^{W_\lambda^\circ}}S_{T,\lambda}$ is a Soergel bimodule (see \cite[Lemma 10.4.9]{gouttardPerverseMonodromicSheaves2021}) we have an isomorphism of $S_{T,\lambda}\otimes_{S_{T,\lambda}^{W_\lambda^\circ}}S_{T,\lambda}\otimes_{S_{T,\lambda}^{W_\lambda^\circ}}S_{T,\lambda}$-modules
\begin{multline*}
M\otimes_{S_{T,\lambda}^{W_\lambda^\circ}}S_{T,\lambda}\cong M\otimes_{S_{T,\lambda}}(S_{T,\lambda}\otimes_{S_{T,\lambda}^{W_\lambda^\circ}}S_{T,\lambda})  \xrightarrow{\sim}\\(S_{T,\lambda}\otimes_{S_{T,\lambda}^{W_\lambda^\circ}}S_{T,\lambda})\otimes_{S_{T,\lambda}}M\cong S_{T,\lambda}\otimes_{S_{T,\lambda}^{W_\lambda^\circ}}M.
\end{multline*}
Moreover, this map satisfies the cocycle condition for the composition 
\begin{multline*}
  M\otimes_{S_{T,\lambda}^{W_\lambda^\circ}}S_{T,\lambda}\otimes_{S_{T,\lambda}^{W_\lambda^\circ}}S_{T,\lambda}\xrightarrow{\sim}S_{T,\lambda}\otimes_{S_{T,\lambda}^{W_\lambda^\circ}}M\otimes_{S_{T,\lambda}^{W_\lambda^\circ}}S_{T,\lambda}\xrightarrow{\sim}\\
  S_{T,\lambda}\otimes_{S_{T,\lambda}^{W_\lambda^\circ}}S_{T,\lambda}\otimes_{S_{T,\lambda}^{W_\lambda^\circ}}M.
\end{multline*}
This ensures the descent of $M$ to an $S_{T,\lambda}^{W_\lambda^\circ}$-module. Finally, the central morphism with $N$ being the diagonal bimodule over $S_{T,\lambda}$ with the right action twisted by an element $\beta\in W_{\lambda,\lambda}$ provides the action of $W_{\lambda,\lambda}$ on $M$.
\end{proof}
Assume that $\F \in D^b(T)$ satisfies one of the equivalent conditions of Proposition \ref{sec:dg-models-monodromic}. Then there is a $W_{\lambda,\lambda}$-equiavriant structure on $\F$ coming from the $W_{\lambda,\lambda}$-action on $M$. As above, let $\mathfrak{o}=W\lambda$ and let $\mathrm{av}_\mathfrak{o}\mathcal{F}\in \bigoplus_{\mu\in\mathfrak{o}}{}_\mu\mathcal{P}_{w_0,\mu}$ be the sum of $w^*\mathcal{F}$ with $w$ being a set of representatives of the coset $W/W_{\lambda,\lambda}$. This sheaf admits a natural $W$-equivariant structure. 

This gives a $W$-action on  $\chi (i_*\mathrm{av}_\mathfrak{o}\F) = \operatorname{Ind}_B^G(\mathrm{av}_\mathfrak{o}\F)$, where $\operatorname{Ind}_B^G(-)$ stands for the functor of parabolic induction. We have the following Corollary of the results in this section, proved in \cite[Theorem 6.1]{chenConjectureBravermanKazhdan2022} with a restriction on characteristic. 
\begin{corollary}
The sheaf $\hc(\operatorname{Ind}_B^G(\mathrm{av}_\mathfrak{o}\F)^W)$, where $(-)^W$ stands for the functor of $W$-invariants, is supported on the preimage of the diagonal under the quotient map $Y^2 \to (G/B)^2$.  
\end{corollary}
\begin{proof}
  It is enough to prove the statement in settings \ref{set:c}-\ref{set:adic}.  By Proposition \ref{sec:dg-models-monodromic} and  Theorem \ref{sec:abel-categ-char}, we see that 
  there is an object $\mathcal{E}$ in $\mathfrak{C}$ with $\hc(\mathcal{E}) \simeq \mathrm{av}_\mathfrak{o}\F$. We have \[\operatorname{Ind}_B^G\mathrm{av}_\mathfrak{o}\F \simeq \chi (i_*\mathrm{av}_\mathfrak{o}\F) \simeq \chi\circ\hc(\mathcal{E}) \simeq \Spgr\star\mathcal{E},\]
  where the first isomorphism holds because $i_*\mathrm{av}_\mathfrak{o}\F$ is supported over the diagonal in $(G/B)^2$, and the last one from Lemma \ref{sec:springer-comonad-4}. We get that $\mathcal{E} \simeq (\operatorname{Ind}_B^G(\mathrm{av}_\mathfrak{o}\F))^W$, where invariants are taken with respect to the $W$-action on the Springer sheaf, where the action is chosen so that $\delta_e \simeq \Spgr^W$. The compatibility of this action and the action used in the formulation of the Proposition is proved in \cite[Proposition 5.4.6]{hhh}.
\end{proof}
\appendix
\section{Sheaves on Vinberg semigroup}
\label{sec:sheav-vinb-semigr}
The aim of this appendix is to give a different proof of Proposition \ref{sec:some-facts-about-1}, namely the fact that the convolution $\star_1$ is left t-exact with respect to the t-structure $(D^{\leq 0}_{w_0}, D^{\geq 0}_{w_0})$, following an idea of Drinfeld.
\subsection{General facts about convolution on semigroups.}
  Let $V$ be an algebraic semigroup, and let $Z \subset V$ be a closed two-sided ideal (i.e. $Z$ is a closed algebraic subset of $V$ which is a two-sided ideal with respect to the semigroup product). Let $\mathring{V} \subset V$ be the open complement of $Z$. 

  The category $D^b(V)$ is semigroupal (monoidal if $V$ is unital), with convolution given by the formula
  \[
    \F \star_V \mathcal{G} := m_!(\F \boxtimes \mathcal{G}),
  \]
  where $m$ stands for the multiplication map $m: V \times V \to V$.

  Since the map $m$ is affine, the convolution $\star_V$ is left t-exact with respect to the perverse t-structure, by Artin's theorem. 

  Write $C \subset V \times V$ for the preimage $m^{-1}(\mathring{V})$ of $\mathring{V}$ under the multiplication map. Not that, since $Z$ is a two-sided ideal, $C \subset \mathring{V} \times \mathring{V}$.
 
  Consider a diagram
  \begin{equation}
    \label{eq:16}
  \begin{tikzcd}
    & C \arrow{r}{m'} \arrow{dl}{\pi'_1} \arrow[dr, "\pi'_2"'] & \mathring{V}  \\
    \mathring{V} & & \mathring{V}
  \end{tikzcd}
\end{equation}
  where $\pi'_1, \pi'_2, m'$ are restrictions of projections $\pi_1,\pi_2:V^2 \to V$ and multiplication map $m:\mathring{V}^2 \to V$ to $C$, respectively. 

  The category $D^b(\mathring{V})$ is equipped with "convolution" operation
  \begin{equation}
    \label{eq:17}
   \F \ustar \mathcal{G} := m'_!(\pi_1'^*\F \otimes \pi_2'^*\G).
\end{equation} 
 
  We will use the following:
  \begin{lemma}
    \label{sec:some-facts-about-2}
    The operation $\ustar$ is left t-exact with respect to the perverse t-structure.
  \end{lemma}
  \begin{proof}
    Let $j:\mathring{V} \to V, i: Z \to V$ be the open and complementary closed embeddings. Using base change and straightforward diagram chase, one obtains an isomorphism
    \begin{equation}
      \label{eq:5}
      \F \ustar \G  \rsim j^*(j_!\F \star_V j_! \G).
    \end{equation}
    Consider the canonical distinguished triangles
    \[
      i_*i^!j_!\F \to j_!\F \to j_*\F \to i_*i^!j_!\F[1],
    \]
    \[
      i_*i^!j_!\G \to j_!\G \to j_*\G \to i_*i^!j_!\G[1].
    \]
    Since $Z$ is a two-sided monoidal ideal, sheaves
    \[
      i_*i^!j_!\F \star_V j_*\G,  j_*\F \star_V i_*i^!j_!\G, i_*i^!j_!\F \star_V i_*i^!j_!\G
    \]
    are supported on $Z$. It follows that we have an isomorphism
    \[
      j^*(j_!\F \star_V j_!\G) \rsim j^*(j_*\F \star_V j_*\G).
    \]
    Now note that functors since $j_*, \star_V$ are left t-exact and $j^*$ is $t$-exact, we get the result from the isomorphism \eqref{eq:5}.
  \end{proof}
\subsection{Vinberg semigroup and the horocycle space.}
We recall some facts about the Vinberg semigroup, following mostly the exposition of \cite{drinfeldGeometricConstantTerm2016b}, and refer the reader there for further bibliography.

Let $G, B, U, T$ be as in the main text, and let $U^-$ stand for a maximal unipotent subgroup opposite to $U$. Let $Z(G)$ stand for the center of $G$. Attached to $G$ there is an affine semigroup $\barGe$, whose group of invertible elements is $\Ge = (G \times T)/Z(G)$. $G$ embeds into $\Ge$, and so $G \times G$ acts on $\barGe$ by multiplication on left and right: $(g_1, g_2)\cdot x = g_1xg_2^{-1}$. In loc. cit., an open subset $\obarGe$ in $\barGe$, such that the complement $\barGe \backslash \obarGe$ is a two-sided ideal, is defined. Let $r$ stand for the semisimple rank of $G$. Let \[\Ta = T/Z(G),\enskip \barTa \simeq \mathbb{A}^r \supset \Ta.\] The semigroup $\barGe$ comes equipped with a homomorphism of semigroups \[\bar{\pi}: \barGe \to \barTa\] and a section \[\bar{\mathfrak{s}}: \barTa \to \obarGe \subset \barGe.\]    

Let $V$ be the fiber over $0 \in \barTa \simeq \mathbb{A}^r$. This is an affine semigroup. Let $\mathring{V} = \obarGe \cap V$, open subset of $\mathring{V}$. Since both $V$ and $\barGe \backslash \obarGe$ are two-sided ideals in $\barGe$, we get that $Z := V \backslash \mathring{V}$ is a two-sided ideal in $V$, so that we are in the situation of previous subsection and Lemma \ref{sec:some-facts-about-2}. 

The proposition below follows directly from the constructions in the \cite[Appendix D]{drinfeldGeometricConstantTerm2016b}.
\begin{proposition}
  \label{sec:vinb-semigr-horocycl}
  \begin{enumerate}[label=\alph*),ref=(\alph*)]
  \item \label{item:14} Open subset $\mathring{V} \subset V$ is the $G \times G$ orbit of $\bar{\mathfrak{s}}(0)$.
  \item \label{item:15} Stabilizer of $\bar{\mathfrak{s}}(0)$ in $G \times G$ is $U \times U^- \times T$, embedded as $(u,u',t) \mapsto (ut,u't)$, so that \[\mathring{V} \simeq \frac{G/U\times G/U^-}{T},\] 
    with $T$ acting on $G/U \times G/U^-$ diagonally on the right.
  \item \label{item:16} The preimage $C$ of $\mathring{V}$ in $\mathring{V} \times \mathring{V}$ under the multiplication map is identified with
    \[
      C = \{((x_1U,x_2U^{-})T,(x_3U,x_4U^-)T) \in \mathring{V}^2 : x_2^{-1}x_3 \in U^{-}B\}.
    \]
  \end{enumerate}
\end{proposition}
  Note that for $V, \mathring{V}, C$ as above, all arrows in the diagram \eqref{eq:16} are $G$-equivariant. Combining Proposition \ref{sec:vinb-semigr-horocycl} with Lemma \ref{sec:some-facts-about-2}, we obtain 
  \begin{corollary}
    \label{sec:vinb-semigr-horocycl-1}
    Operation $\ustar$, defined on the derived category \[D^b(G\backslash (G/U \times G/U^-)/T)\] via the formula \eqref{eq:17}, is left t-exact with respect to the perverse t-structure.  
  \end{corollary}
\subsection{Radon transform and convolution on monodromic categories.}
  We introduce some more notation. Let $H, H_1, H_2 \in \{U, U^{-}\}$, and consider a diagram
\[
  \begin{tikzcd}
    & G/H_1 \times G/H \times G/H_2 \arrow[ld,"p_{23}"']\arrow[rd,"p_{12}"]\arrow[r,"p_{13}"]  & G/H_1 \times G/H_2 \\
G/H_1 \times G/H &                                      & G/H \times G/H_2
  \end{tikzcd}
\]
Here $p_{ij}$ stands for the projection along factors $i,j$. Note that all arrows in the above diagram are equivariant with respect to the left diagonal action of $G$ and right diagonal action of $T$. We define the convolution operation $-\bistar{H}-$: for \[\F \in D^b(G \backslash (G/H_1 \times G/H)/T), \G \in D^b(G\backslash (G/H \times G/H_2)/T)\] write
\[
  \F\bistar{H}\G := p_{13!}(p_{12}^*\F\otimes p_{23}^*\G).
\]
Note that for $H_1 = H_2 = H = U$ this coincides with the monoidal structure $- \star_1 -$. Operations $\bistar{H}$ define an associative product on the sum of categories,
\[
  \bigoplus_{H_1, H_2 \in \{U, U^-\}}D^b(G\backslash (G/H_1 \times G/H_2)/T).
\]
where the product is defined to be zero when the definition of neither $\bistar{H}$ is applicable. 

 Let \[O_{w_0} = \{(xU^-, yU)T \in (G/U^- \times G/U)/T: x^{-1}y \in U^{-}U\}.\]
The closed embedding $\iota_{w_0}:O_{w_0} \to (G/U^- \times G/U)/T$ is $G$-equivariant. Define
\[
  \mathfrak{R} = \iota_{w_0!}\cc_{G \backslash O_{w_0}} \simeq \iota_{w_0*}\cc_{G \backslash O_{w_0}}.
\]

Let \[\tilde{O}_{w_0} = \{(xU^-, yU)T \in (G/U^- \times G/U)/T: x^{-1}y \in U^{-}B\}.\]
Let $\phi: \tilde{O}_{w_0} \to T$ stand for the projection \[(xU^-, yU) \mapsto x^{-1}y = u^{-}ut \mapsto t, u \in U, u^- \in U^-, t \in T,\] which is easy to see to be well-defined, $G$-equivariant (with trivial action of $G$ on $T$) and descends to the quotient by the right diagonal action of $T$. This trivializes $T$-torsor \[\tilde{O}_{w_0} \simeq O_{w_0}\times T.\] Write $\hat{\mathcal{L}}_\lambda$ for the pro-unit in the $\lambda$-monodromic category on $T$ and, abusing notation, its pullback to $\tilde{O}_{w_0}$ with perverse shift. 

The open embedding $\tilde{\iota}_{w_0}:\tilde{O}_{w_0} \to (G/U^- \times G/U)/T$ is $G$-equivariant. Define
\[
 \hat{\mathfrak{R}}_! = \tilde{\iota}_{w_0!}\hat{\mathcal{L}}_\lambda \in \pro D^b(G\backslash (G/U^- \times G/U)/T).
\]

Similarly, swap $U$ and $U^{-}$ in the above and define 
\[
 \hat{\mathfrak{R}}_* = \tilde{\iota}_{w_0*}\hat{\mathcal{L}}_\lambda\in \pro D^b(G\backslash (G/U \times G/U^{-})/T).
\]
We call the functor
\[
  \F \mapsto \F\bistar{U^-}\mathfrak{R}:  D^b(G\backslash (G/U \times G/U^{-})/T) \to D^b(G\backslash (G/U \times G/U)/T),
\]
the (!-version of) Radon transform functor. 

We record some standard properties of Radon transform we will need. By monodromic categories we mean, as before, categories of sheaves monodromic with respect to projections $G/U \to G/B$ and $G/U^{-}\to G/B^{-}$, where $B^-$ is the Borel subgroup containing $U^-$. Recall that, when treating monodromic categories, we shift the convolution product by $[\dim T]$.

\begin{proposition}
  \label{sec:radon-transf-conv}
  \begin{enumerate}[label=\alph*),ref=(\alph*)]

  \item \label{item:17} There is an isomorphism of functors, natural in $\F, \G:$
    \[
      \F \ustar \G \rsim \F \bistar{U^-} \mathfrak{R} \bistar{U} \G.
    \]
  \item\label{item:21} When restricted to monodromic categories, there is an isomorphism of functors natural in $\F:$
    \[
      \F \bistar{U^-} \mathfrak{R} \rsim \F \bistar{U^-} \hat{\mathfrak{R}}_!.
    \]
  \item\label{item:22} When restricted to monodromic categories, the Radon transform functor is an equivalence, with inverse given by
    \[
      \F \mapsto \F\bistar{U} \hat{\mathfrak{R}}_*.
    \]
  \end{enumerate}
  \begin{proof}
    Part \ref{item:17}   is a straightforward computation using the base change isomorphism and the isomorphism
    \[
      C \rsim \{(x_1U, x_2U^-, x_3U, x_4U^-)T \in (G/U \times G/U^-)^2/T : x_2^{-1}x_3 \in U^-U,\}
    \]
    where a single copy of $T$ acts diagonally on the right. The isomorphism is given by the formula
    \begin{multline}
      ((x_1U,x_2U^{-})T,(x_3U,x_4U^-)T) \mapsto \\
      \mapsto (x_1U,x_2U^{-1},x_3\phi(x_2U^-,x_3U)^{-1}U, x_4\phi(x_2U^-,x_3U)^{-1}U^-)T.
\end{multline}
    
   Part \ref{item:21} is a straightforward computation using base change and the fact that $\hat{\mathcal{L}}$ is the monoidal unit in the monodromic category of sheaves in $D^b(T)$, see references in Proposition \ref{braid_relations}.

   Part \ref{item:22} is completely analogous to \ref{braid_relations} \ref{item:20}. 
  \end{proof}
\end{proposition}
\begin{bremark}
 Without monodromicity assumption, the functor adjoint to the Radon transfrom is given by the appropriate $*$-convolution with object analogous to $\mathfrak{R}$, as opposed to $!$-convolution with a $*$-extension of a sheaf on an open subset, as in the proposition above.   
\end{bremark}
Corollary \ref{sec:some-facts-about-1} now follows directly from Corollary \ref{sec:vinb-semigr-horocycl-1} and the following
\begin{lemma}
  The functor $(-)\bistar{U} \hat{\mathfrak{R}}_*$ intertwines the products $\star_1 = \bistar{U}$ and $\ustar$ on the monodromic categories. Namely, for any $\F, \G \in \Hecke{1}_\lambda$, there is a natural isomorphism
  \[
    (\F \star_1 \G)\bistar{U} \hat{\mathfrak{R}}_* \simeq (\F\bistar{U} \hat{\mathfrak{R}}_*) \ustar (\G\bistar{U} \hat{\mathfrak{R}}_*).
  \]
\end{lemma}
\begin{proof}
  By Proposition \ref{sec:radon-transf-conv} \ref{item:17} and \ref{item:21} we have an isomorphism
\[(\F\bistar{U} \hat{\mathfrak{R}}_*) \ustar (\G\bistar{U} \hat{\mathfrak{R}}_*) \simeq (\F\bistar{U} \hat{\mathfrak{R}}_*)  \bistar{U^-} \hat{\mathfrak{R}}_! \bistar{U}(\G\bistar{U} \hat{\mathfrak{R}}_*).\]
Using the associativity of $\star$-operations and Proposition \ref{sec:radon-transf-conv} \ref{item:22} we get
\begin{multline}
  (\F\bistar{U} \hat{\mathfrak{R}}_*)  \bistar{U^-} \hat{\mathfrak{R}}_! \bistar{U}(\G\bistar{U} \hat{\mathfrak{R}}_*) \simeq \\
 \simeq (\F\bistar{U} \hat{\mathfrak{R}}_*  \bistar{U^-} \hat{\mathfrak{R}}_!) \bistar{U}(\G\bistar{U} \hat{\mathfrak{R}}_*) \simeq \\
 \simeq (\F \star_1 \G)\bistar{U} \hat{\mathfrak{R}}_*.
\end{multline}
and so the result.
\end{proof}
\bibliography{bibliography}

\begin{thebibliography}{EGNO16}

\bibitem[AB20]{arkhipovPerverseSheavesAffine2020}
Sergey Arkhipov and Roman Bezrukavnikov.
\newblock Perverse sheaves on affine flags and {{Langlands}} dual group.
\newblock {\em arXiv:math/0201073}, February 2020.

\bibitem[AHJR14]{acharWeylGroupActions2014}
Pramod~N. Achar, Anthony Henderson, Daniel Juteau, and Simon Riche.
\newblock Weyl group actions on the {{Springer}} sheaf.
\newblock {\em Proceedings of the London Mathematical Society},
  108(6):1501--1528, June 2014.

\bibitem[Bal11]{balmerSeparabilityTriangulatedCategories2011}
Paul Balmer.
\newblock Separability and triangulated categories.
\newblock {\em Advances in Mathematics}, 226(5):4352--4372, March 2011.

\bibitem[BD13]{boyarchenkoCharacterSheavesUnipotent2013}
Mitya Boyarchenko and Vladimir Drinfeld.
\newblock Character sheaves on unipotent groups in positive characteristic:
  Foundations.
\newblock {\em arXiv:0810.0794 [math]}, January 2013.

\bibitem[Bez16]{bezrukavnikovTwoGeometricRealizations2016}
Roman Bezrukavnikov.
\newblock On two geometric realizations of an affine {{Hecke}} algebra.
\newblock {\em Publications math\'ematiques de l'IH\'ES}, 123(1):1--67, 2016.

\bibitem[BFO12]{bezrukavnikovCharacterDmodulesDrinfeld2012}
Roman Bezrukavnikov, Michael Finkelberg, and Victor Ostrik.
\newblock Character {{D-modules}} via {{Drinfeld}} center of {{Harish-Chandra}}
  bimodules.
\newblock {\em Inventiones mathematicae}, 188(3):589--620, 2012.

\bibitem[BGO18]{ben-zvigunninghamorem2018}
David {Ben-Zvi}, Sam Gunningham, and Hendrick Orem.
\newblock Highest weights for categorical representations.
\newblock {\em International Mathematics Research Notices},
  2020(24):9988--10004, 2018.

\bibitem[BL06]{bernsteinEquivariantSheavesFunctors2006}
Joseph Bernstein and Valery Lunts.
\newblock {\em Equivariant Sheaves and Functors}.
\newblock {Springer}, 2006.

\bibitem[BM81]{borhoRepresentationsGroupesWeyl1981}
Walter Borho and Robert MacPherson.
\newblock Representations des groupes de {{Weyl}} et homologie d'intersection
  pour les variete nilpotents.
\newblock {\em CR Acad. Sci. Paris t.}, 292:707--710, 1981.

\bibitem[BN09]{ben-zviCharacterTheoryComplex2009}
David {Ben-Zvi} and David Nadler.
\newblock The character theory of a complex group.
\newblock {\em arXiv preprint arXiv:0904.1247}, 2009.

\bibitem[BR18]{bezrukavnikovTopologicalApproachSoergel2018}
Roman Bezrukavnikov and Simon Riche.
\newblock A topological approach to {{Soergel}} theory.
\newblock {\em arXiv preprint arXiv:1807.07614}, 2018.

\bibitem[BT22]{hhh}
Roman Bezrukavnikov and Kostiantyn Tolmachov.
\newblock Monodromic model for {{Khovanov}}\textendash{{Rozansky}} homology.
\newblock {\em Journal f\"ur die reine und angewandte Mathematik (Crelles
  Journal)}, April 2022.

\bibitem[BY13]{bezrukavnikovKoszulDualityKacMoody2013}
Roman Bezrukavnikov and Zhiwei Yun.
\newblock On {{Koszul}} duality for {{Kac-Moody}} groups.
\newblock {\em Representation Theory of the American Mathematical Society},
  17(1):1--98, 2013.

\bibitem[Che22]{chenConjectureBravermanKazhdan2022}
Tsao-Hsien Chen.
\newblock On a conjecture of {{Braverman-Kazhdan}}.
\newblock {\em Journal of the American Mathematical Society}, 35(4):1171--1214,
  October 2022.

\bibitem[CYD17]{chenFormulaGeometricJacquet2017}
Tsao-Hsien Chen and Alexander Yom~Din.
\newblock A formula for the geometric {{Jacquet}} functor and its character
  sheaf analogue.
\newblock {\em Geometric and Functional Analysis}, 27(4):772--797, 2017.

\bibitem[DG16]{drinfeldGeometricConstantTerm2016b}
Vladimir Drinfeld and Dennis Gaitsgory.
\newblock Geometric constant term functor(s).
\newblock {\em Selecta Mathematica}, 22(4):1881--1951, October 2016.

\bibitem[DS18]{dellambrogioNoteTriangulatedMonads2018}
Ivo Dell'Ambrogio and Beren Sanders.
\newblock A note on triangulated monads and categories of module spectra.
\newblock {\em Comptes Rendus Mathematique}, 356(8):839--842, August 2018.

\bibitem[EGNO16]{etingofTensorCategories2016}
Pavel Etingof, Shlomo Gelaki, Dmitri Nikshych, and Victor Ostrik.
\newblock {\em Tensor {{Categories}}}.
\newblock {American Mathematical Soc.}, August 2016.

\bibitem[GH22]{gorskyHilbertSchemesIfication2022}
Eugene Gorsky and Matthew Hogancamp.
\newblock Hilbert schemes and y\textendash ification of
  {{Khovanov}}\textendash{{Rozansky}} homology.
\newblock {\em Geometry \& Topology}, 26(2):587--678, June 2022.

\bibitem[Gin89]{ginzburgAdmissibleModulesSymmetric1989}
Victor Ginzburg.
\newblock Admissible modules on a symmetric space.
\newblock {\em Ast\'erisque}, (173-74):199--255, 1989.

\bibitem[Gou21]{gouttardPerverseMonodromicSheaves2021}
Valentin Gouttard.
\newblock {\em Perverse {{Monodromic Sheaves}}}.
\newblock PhD thesis, Universit\'e Clermont Auvergne, July 2021.

\bibitem[HL23]{holicharactersheveshomfly2023}
Quoc~P. Ho and Penghui Li.
\newblock Graded character sheaves, {HOMFLY-PT} homology, and {H}ilbert schemes
  of points on $\mathbb{C}^2$.
\newblock {\em arXiv preprint arXiv:2305.01306}, 2023.

\bibitem[LO08]{laszloSixOperationsSheaves2008a}
Yves Laszlo and Martin Olsson.
\newblock The six operations for sheaves on {{Artin}} stacks {{II}}: Adic
  coefficients.
\newblock {\em Publications Math\'ematiques de l'IH\'ES}, 107:169--210, 2008.

\bibitem[Lus85]{lusztigCharacterSheaves1985}
George Lusztig.
\newblock Character sheaves {{I}}.
\newblock {\em Advances in Mathematics}, 56(3):193--237, 1985.

\bibitem[Lus14]{lusztigTruncatedConvolutionCharacter2014}
George Lusztig.
\newblock Truncated convolution of character sheaves.
\newblock {\em arXiv:1308.1082 [math]}, February 2014.

\bibitem[Lus15]{lusztigNonunipotentCharacterSheaves2015}
George Lusztig.
\newblock Non-unipotent character sheaves as a categorical centre.
\newblock {\em arXiv preprint arXiv:1506.04598}, 2015.

\bibitem[LY20]{lusztigEndoscopyHeckeCategories2020a}
George Lusztig and Zhiwei Yun.
\newblock Endoscopy for {{Hecke}} categories, character sheaves and
  representations.
\newblock {\em Forum of Mathematics, Pi}, 8:e12, 2020.

\bibitem[MV88]{mirkovicCharacteristicVarietiesCharacter1988}
Ivan Mirkovi{\'c} and Kari Vilonen.
\newblock Characteristic varieties of character sheaves.
\newblock {\em Inventiones mathematicae}, 93(2):405--418, 1988.

\end{thebibliography}
\bibliographystyle{alpha}

\Addresses
\end{document}